\documentclass[12pt]{amsart}
\usepackage{amsmath,amssymb,amsbsy,amsfonts,latexsym,amsopn,amstext,cite,
                                               amsxtra,euscript,amscd,bm}

\usepackage{mathrsfs}
\usepackage{enumitem}
\usepackage{color}
\usepackage{marginnote}

\usepackage{ifthen}
\newboolean{showcomments}
\setboolean{showcomments}{true}

\usepackage{color}
\usepackage{float}
\usepackage{epstopdf}

\usepackage[np]{numprint}
\npdecimalsign{\ensuremath{.}}

\usepackage{bibentry}

\usepackage[english]{babel}
\usepackage{mathtools}
\usepackage{url}
\usepackage[norefs,nocites]{refcheck}

\newtheorem{theorem}{Theorem}
\newtheorem{lemma}[theorem]{Lemma}

\newtheorem{cor}[theorem]{Corollary}

\newtheorem{prop}[theorem]{Proposition}

\numberwithin{equation}{section}
\numberwithin{theorem}{section}
\numberwithin{table}{section}
\numberwithin{figure}{section}

\def\squareforqed{\hbox{\rlap{$\sqcap$}$\sqcup$}}
\def\qed{\ifmmode\squareforqed\else{\unskip\nobreak\hfil
\penalty50\hskip1em \nobreak\hfil\squareforqed
\parfillskip=0pt\finalhyphendemerits=0\endgraf}\fi}

\newfont{\teneufm}{eufm10}
\newfont{\seveneufm}{eufm7}
\newfont{\fiveeufm}{eufm5}
\newfam\eufmfam
     \textfont\eufmfam=\teneufm
\scriptfont\eufmfam=\seveneufm
     \scriptscriptfont\eufmfam=\fiveeufm
\newcommand{\bmu}{{\mu}}

\def\eqref#1{(\ref{#1})}

\def\le{\leqslant}

\def\ge{\geqslant}

\def\v0{\mathbf 0}

\def\e{{\mathbf{\,e}}}

 \def\0{{\mathbf{0}}}

\def\({\left(}
\def\){\right)}

\newcommand{\commentby}[3]{%
  \ifthenelse{\boolean{showcomments}}%
    {\textcolor{#1}{[\textbf{#2}: #3]}}%
    {}%
}

\def \N{{\mathbb N}}
\def \Z{{\mathbb Z}}

\def \C{{\mathbb C}}

 \def\1{\,\rlap{\mbox{\small\rm 1}}\kern.15em 1}
 
 \def\build#1_#2^#3{\mathrel{\mathop{\kern 0pt#1}\limits_{#2}^{#3}}}
 \def\tend#1#2{\build\hbox to 12mm{\rightarrowfill}_{#1\rightarrow #2}^{ }}

 \def\tendN{\tend{N}{\infty}}
 \def\converge#1#2#3#4{\build\hbox to
 	#1mm{\rightarrowfill}_{#2\rightarrow #3}^{\hbox{\scriptsize #4}}}
 \def\Ima#1{{\rm {Im}}({#1})}
 \def\Rea#1{{\rm {Re}}({#1})}

 \newcommand{\ds}{\displaystyle}

 \newtheorem{sarnak}{Sarnak's conjecture}

\begin{document}

\title[Operator ergodic theorems with M\"obius "weights"]
{Operator ergodic theorems with M\"obius "weights"}

\author[e.~H.~el Abdalaoui]{ {\MakeLowercase{el}}  Houcein  {\MakeLowercase{el}}  Abdalaoui}
\address{Laboratoire de Math{\'e}matiques Rapha{\"e}l Salem, UMR 5085 CNRS,
Universit{\'e} de Rouen-Normandie,
F76801 Saint-{\'E}tienne-du-Rouvray, France}
\email{elhoucein.elabdalaoui@univ-rouen.fr}

 \author[M. Lin] {Michael Lin}
\address{Department of Mathematics, Ben-Gurion University of the Negev, Beer-Sheva, Israel}
\email{lin@math.bgu.ac.il}

\dedicatory{\large Dedicated to the memory of Alexandra Bellow}

\begin{abstract}
Motivated by Sarnak's conjecture in topological dynamics for the M\"obius function $\mu$,
we study, for a power-bounded $T$ on a Banach space $E$, the weak convergence
$$
(*) \qquad \qquad  \frac1N\sum_{n=1}^N \mu(n)T^nv  \to 0 \text{ weakly } \forall v\in E.
$$

For that, we introduce a notion of dynamical entropy for operators, which we denote 
$h^*_{top}(T)$, and show that if Sarnak's conjecture is true, then $h^*_{top}(T)=0$ 
implies the desired convergence (*). We conclude an equivalent operator formulation of 
Sarnak's conjecture.
	
For several classes of operators we prove that (*) holds, and that $h^*_{top}(T)=0$.
\end{abstract}

\keywords{ modulated ergodic theorems, M\"obius function, Sarnak's conjecture,
topological entropy, operator entropy, power-bounded operators}
\subjclass[2020]{Primary: 47A35, 37B40, 11N37; Secondary: 11L07, 37A44}

\maketitle

\tableofcontents

\section{Introduction}

\subsection{Background }

The M\"{o}bius function  is defined for the positive integers $n$ by
\begin{equation}\label{Mobius}
\bmu(n)= \begin{cases}
(-1)^r\ {\rm {~if~}} n {\rm {~is~the~product~of~}} r {\rm {~distinct~primes}}; \\
0\  {\rm {~otherwise, i.e.~}} n {\rm {~has~a~square~factor}};
\end{cases}
\end{equation}
This function is of great importance in number theory because of its  connection with
 the Riemann $\zeta$-function, via the formulae (see \cite[Theorems 300, 287 and 302]{HW})
\[
\sum_{n=1}^{+\infty}\frac{\bmu(n)}{n^s}=\frac{1}{\zeta(s)}, \qquad \text{and} \qquad
\sum_{n=1}^{+\infty}\frac{|\bmu(n)|}{n^s}=\frac{\zeta(s)}{\zeta(2s)}\ ,
\qquad \Rea{s}>1.
\]
Furthermore,   the estimate
\[
\left|\ds \sum_{n \le t}\bmu(n)\right|=O_\varepsilon\left(t^{\frac12+\varepsilon}\right)\qquad
{\rm as} \quad  t \longrightarrow +\infty,\quad \forall \varepsilon >0,
\]
is equivalent to the Riemann Hypothesis (\cite[p. 370]{Titchmarsh}).
\medskip

The M\"{o}bius function satisfies
$\frac{1}{t}\sum_{n \leq t}\mu(n) \tend{t}{+\infty}0$; this follows from Kronecker's 
lemma combined with von Mangoldt's result \cite{vM} (see also  \cite{Lan}):
\[
\sum_{n \geq 1}\frac{\mu (n)}{n} =0.
\]
This last relation also implies that $\frac{1}{t}\sum_{n \leq t}\mu(n) \tend{t}{\infty}0$
 is equivalent to the prime number theorem \cite{Lan1}.

We note that $\lim_{N\to \infty} \frac1N\sum_{n=1}^N |\mu(n)| =\frac6{\pi^2}$
\cite[p. 270, Theorem 334]{HW}. Consequently, $\sum_{n=1}^\infty \frac{|\mu(n)|}n=\infty$.

\subsection{The M\"obius randomness law}

The dynamical version of the M\"{o}bius randomness law was  initiated by
P. Sarnak in \cite{sarnak1},\cite{sarnak}.
\smallskip

In this work a {\it topological dynamical system} is a pair  $(X,\tau)$, with $X$  a compact
Hausdorff space and $\tau$ a continuous map from $X$ into $X$ (metrizability of $X$ or 
invertibilty of $\tau$ are not assumed).
\smallskip

A  bounded sequence $(a_n)_{n\in\mathbb N}$ is called {\it deterministic} if there is 
 a  topological dynamical system $(X,\tau)$ with zero topological entropy
(called a {\it deterministic} system), such that for some continuous function 
$f \in C_\C (X)$  and some point $x\in X$, $\ a_n=f(\tau^n x)$, $\forall n\ge 0$.
 
 Sarnak made the following conjecture, which he called "the M\"obius randomness law".

\begin{sarnak}\label{conj-sarnak}
The M\"obius function is orthogonal to any deterministic sequence $(a_n)_{n \in \N}$,
that is,
	\begin{eqnarray}\label{sarnak-conj}
	\frac{1}{N}\sum_{n=1}^{N} \mu(n)a_n \tendN 0.
	\end{eqnarray}
\end{sarnak}
An equivalent formulation of Sarnak's conjecture is:
\begin{sarnak}
    For {\bf every} topological
dynamical system $(X,\tau)$ with zero topological entropy, the operator $Tf=f\circ \tau$,
$f \in C(X)$, satisfies
\begin{equation} \label{sarnak-op}
\frac1N\sum_{n=1}^{N} \mu(n) T^nf \to 0 \quad \text{\rm weakly in }\ C(X),\
	\forall f\in C(X),
\end{equation}
which  is equivalent to
\begin{equation} \label{sarnak-op1}
\frac1N\sum_{n=1}^{N} \mu(n) f(\tau^nx) \to 0 \quad
\forall f \in  C(X) \ \text{ and every } \ x \in X.
\end{equation}
\end{sarnak}

{\bf Remark.} If \eqref{sarnak-op1} holds for every continuous real function then it does also
for every continuous complex function.
\smallskip

The importance of Sarnak's conjecture is that if it fails, then the long-standing Chowla's
conjecture, believed to hold true by most number theorists, also fails \cite{sarnak}.
\smallskip

It was proved in \cite[Corollary 10]{AKLR} that if Sarnak's conjecture  (in the above
formulation) is true {\it for all zero entropy systems}, then \eqref{sarnak-op}
and \eqref{sarnak-op1} can be strengthened to
\begin{equation} \label{sarnak-norm}
\Big\|\frac1N\sum_{n=1}^{N} \mu(n) T^nf\Big\|_\infty \to 0 \quad \forall f\in C_\mathbb C(X).
\end{equation}

\medskip

The orthogonality of the M\"obius function to any sequence arising from a rotation
dynamical system ($X$ is the unit circle $\mathbb T$, identified with $[0,1)$ mod 1,
 and $Tx=x+\alpha \mod 1$, $\ \alpha \in [0,1)$) follows
from the following inequality of Davenport \cite{Da}
\begin{equation} \label{davenport}
\max_{\theta \in [0,1)}\left|\displaystyle\sum_{k \leq t}\mu(k)e^{2\pi ik\theta}\right| \le
 C_\varepsilon \frac{t}{\log(t)^{\varepsilon}}, \qquad \text{\rm for any}\  \varepsilon >0.
\end{equation}



\subsection{Th M\"obius randomness law for operators}

Following Veech \cite{V2},\cite{V}, the main goal of this paper is to 
investigate the M\"{o}bius randomness law for  linear operators on Banach spaces.
Given a power-bounded operator $T$ on a Banach space $E$, the problem, motivated by
\eqref{sarnak-op}, is to obtain conditions for the convergence
\begin{equation} \label{sarnak-op-E}
\frac1N\sum_{n=1}^{N} \mu(n) T^nv \to 0 \quad \text{\rm weakly in } \  E, \quad
\text{\rm for every } v \in E.
\end{equation}
We thus introduce notions  of dynamical topological entropy for power-bounded
 operators. We then  show that the original Sarnak's conjecture
implies that \eqref{sarnak-op-E} holds for power-bounded operators for which our
dynamical topological operator entropy is zero.

For some operators, most importantly weakly almost periodic ones, we prove 
\eqref{sarnak-op-E} directly, and show that their dynamical topological operator entropy
is zero. 
\medskip

\section{Modulated ergodic theorems with M\"obius weights} \label{modulated}

Let $T$ be a power-bounded operator ($\sup_{n \ge 0}\|T^n\| < \infty$) on a Banach space $E$ 
and let $(b_n)$ be a sequence of real numbers (complex when $E$ is over $\mathbb C$). 
Lin, Olsen and Tempelman \cite{LOT} studied conditions for the norm convergence of 
$\frac1N\sum_{n=1}^N b_n T^n v$ for every $v\in E$. The question is obviously meaningful 
only when $\limsup_N \frac1N \sum_{n=1}^N |b_n| >0$.
\smallskip

{\bf Definition.}  A bounded operator $T$ on a Banach space $E$ is called {\it weakly 
almost periodic} (WAP) if for every $v \in E$ the orbit $(T^nv)$ is conditionally weakly
 compact. Equivalently, by the Eberlein-Shmulyan theorem, for every $v \in E$ the orbit
 is weakly sequentially compact. By the definition, a WAP operator is power-bounded
By the mean ergodic theorem, a WAP operator is mean ergodic
(the averages $\frac1N\sum_{n=1}^N T^n$ converge in the strong operator topology).
\smallskip

{\bf Remarks.} 1. Power-bounded operators on reflexive Banach spaces are always
weakly almost periodic.

2. If $T$ is WAP on a complex Banach space, then $\lambda T$ is WAP for every $|\lambda|=1$.

3. If $T^n$ converges in the weak operator topology, e.g. the shift $Tf(n)=f(n+1)$ on $c_0$,
 then $T$ is weakly almost periodic.
\smallskip

A weakly almost periodic $T$ on a complex Banach space $E$ induces the
{\it Jacobs-DeLeeuw-Glicksberg decomposition} \cite[Section 2.4]{K} $E=E_0 \oplus E_1$, with
$$
E_0=\{w\in E: T^{n_j}w \to 0 \quad \text{for some } (n_j) \};
$$
$$
E_1=\text{closed lin. span} \{u\in E:  Tu=\lambda u,\ |\lambda|=1 \}.
$$

\begin{theorem} \label{wap}
Let $T$ be weakly almost periodic on a Banach space $E$. Then
\begin{equation} \label{strong}
\Big\| \frac1N\sum_{n=1}^{N} \mu(n) T^nv\Big\|  \to 0 \qquad \text{\rm for every }\ v\in E.
\end{equation}
\end{theorem}
\begin{proof}
We first assume that $E$ is a complex Banach space. As noted in the proof of
\cite[Theorem 1.2]{LOT},  the boundedness of $(\mu(n))$ implies  the convergence \eqref{strong}
for every $v\in E_0$. By \eqref{davenport}, $\frac1N\sum_{n=1}^{N} \mu(n)\lambda^n \to 0$
 for every $\lambda \in \mathbb T$; this yields the convergence \eqref{strong} for the
generators of $E_1$, hence for every $v \in E_1$.
\smallskip

We now prove the theorem for $E$ real, using a complexification technique \cite{MST}. Since
$T$ is a contraction in the equivalent norm $|\|v|\|:=\sup_{n\ge 0}\|T^nv\|$, we assume now
that $T$ is a contraction.  Let $\tilde E=E\times E$ with addition coordinate-wise, and
multiplication by scalars defined  by
$$
(\alpha+i\beta)(u,v):=(\alpha u - \beta v,\beta u+\alpha v) \quad
\forall u,v\in E,\ \forall \alpha,\beta \in \mathbb R.
$$
Then $\tilde E$ is a vector space over $\mathbb C$, conveniently denoted by $E\oplus iE$.
On $\tilde E$ we use the Taylor norm (see \cite[p. 6]{MST})
$$
\|u+iv\|:=\sup_{0\le \theta \le 2\pi} \|u\cos\theta - v\sin\theta\| =
\sup_{\|\phi\|_{E^*}\le 1} \sqrt{\phi(u)^2+\phi(v)^2}.
$$
Then $E$ is isometrically isomorphic (over $\mathbb R$) to $E\times\{0\}$. We define
$\tilde T(u,v)=(Tu,Tv)$; by \cite[Proposition 4]{MST} $\tilde T$ is a contraction on
$\tilde E$, and extends $T$. We prove that $\tilde T$ is weakly almost periodic when $T$ is.
Fix $(u,v) \in \tilde E$. For  $\psi \in\tilde E^*$,  $\ \phi_1(u):=\Rea\psi((u,0))$
and $\phi_2(u):=\Ima\psi((u,0))$ are elements of $E^*$. Then
$$
\psi(((T^nu,0))=\phi_1(T^nu)+i\phi_2(T^nu),\quad \psi((0,T^nv))= i\phi_1(T^nv)-\phi_2(T^nv).
$$
Given a sequence $(n_j)$, there is a subsequence $(\ell_k) \subset(n_j)$ such that
$T^{\ell_k}u$ and $T^{\ell_k}v$ both converge weakly, say to $u_0$ and $v_0$. Hence 
$\psi(\tilde T^{\ell_k}(u,v))=\psi((T^{\ell_k}u,T^{\ell_k}v))$ converges to
$\psi((u_0,v_0))$ for every $\psi \in \tilde E^*$, so $\tilde T^{\ell_k}(u,v)$
converges weakly to $(u_0,v_0)$.
\smallskip

Applying the theorem already proved for the complex case to $\tilde T$ on $\tilde E$, 
we obtain \newline
$\big\| \frac1N\sum_{n=1}^{N} \mu(n) T^nu\big\|  =
\big\| \frac1N\sum_{n=1}^{N} \mu(n)\tilde T^n(u,0)\big\|\to 0 $.
\end{proof}

{\bf Remarks.} 1. Weak convergence of the averages \eqref{sarnak-op-E} when $T$ is WAP
 follows from Veech \cite[Proposition 10.1]{V2}, since for any $\phi \in E^*$ the sequence
$(\phi(T^nv))$ is weakly almost periodic for the shift on $\ell_\infty(\mathbb N)$ when
$T$ is almost periodic
\cite[p. 36]{B}. The special case of $T$ induced on $C(X)$ by a weakly almost periodic
continuous map of a compact metric space $X$ was proved in \cite[Corollary 6.2]{EN}.

2. If $E$ is a Banach space with a basis such that all its power-bounded operators are WAP, 
then,  since WAP operators are mean ergodic, $E$ is reflexive by \cite{FLW}.

\begin{cor} \label{stationary}
Let $T$ be a positive contraction of $L^1$ of a probability space $(X,\Sigma,m)$
satisfying $T{\mathbf 1}={\mathbf 1}$. Then every $f \in L^1(X,m)$ satisfies
\begin{equation} \label{L1}
\Big\| \frac1N\sum_{n=1}^N \mu(n)T^n f \Big\|_1 \to 0.
\end{equation}
\end{cor}
\begin{proof} It is known that $T$ is also a contraction of $L^2(X,m)$ (e.g.
\cite[p. 65]{K}). For $f \in L^2(X,m)$  we have
$$
\Big\| \frac1N\sum_{n=1}^N \mu(n)T^n f \Big\|_1 \le
\Big\| \frac1N\sum_{n=1}^N \mu(n)T^n f \Big\|_2 \to 0,
$$
by Theorem \ref{wap}. Since $L^2$ is dense in $L^1$ and
$\big\| \frac1N\sum_{n=1}^N \mu(n)T^n \Big\|_1 \le 1$ for every $N$,
 \eqref{L1} holds for every $f\in L^1(X,m)$.
\end{proof}

{\bf Remark.} In fact, by Komorn\'\i k \cite[Proposition 1.4]{Kom}, $T$ of Corollary
\ref{stationary} is WAP in $L^1(X,m)$.
\medskip

{\bf Example.} {\it Corollary \ref{stationary} may fail when $m(X)=\infty$.}

\noindent
On $\ell^1(\mathbb Z)$ define $T(f_k)=(f_{k-1})$. Then the unit vectors $e^j$,
defined by $e^j_k=\delta_{jk}$, satisfy $T^ne^j=e^{j+n}$, and $T^ne^0=e^n$.
Let $\phi\in\ell^\infty$ be given by $\phi_n=\mu(n)$ for $n> 0$ and
$\phi_n=0$ for $n \le 0$. Then
$$
\frac1N \sum_{n=1}^N \mu(n)\langle \phi,T^ne^0 \rangle =
\frac1N \sum_{n=1}^N \mu(n) \langle \phi,e^n\rangle =
\frac1N \sum_{n=1}^N \mu(n)^2 \to\ \frac 6{\pi^2} \ne 0.
$$
Hence even \eqref{sarnak-op-E} fails.

\begin{prop}\label{H-contraction}
Let $T$ be a contraction in a Hilbert space $\mathcal H$.
Then $\| \frac1N \sum_{n=1}^N \mu(n) T^n\| \to 0$. Precisely,
$\| \frac1N \sum_{n=1}^N \mu(n) T^n\| \le C_\alpha/(\log N)^\alpha$ for any $\alpha>0$.
\end{prop}
\begin{proof}  We first assume that $\mathcal H$ is over $\mathbb C$.
	
As noted in \cite[Corollary 2.11]{CW}, for $U$ unitary, \eqref{davenport} and the spectral
theorem yield $\| \frac1N \sum_{n=1}^N \mu(n) U^n\| \le C_\alpha/(\log N)^\alpha$ for any
$\alpha> 0$.  For a contraction $T$, the unitary dilation theorem \cite{SN} says that there
exists a Hilbert space $\mathcal K \supset \mathcal H$ and a unitary $U$ on $\mathcal K$
such that $T^n=PU^n$ for $n \ge 0$, where $P$ is the orthogonal projection of $\mathcal K$
onto $\mathcal H$. This proves the theorem in the complex case.

When $\mathcal H$ is a real Hilbert space, we use the  vector space complexification
$\tilde{\mathcal H}=\mathcal H \times \mathcal H$ with the norm
$\|u+iv\|_H^2:=\|u\|^2+\|v\|^2$ \cite[Proposition 5]{MST}.  Then $\tilde{\mathcal H}$ is
a complex Hilbert space \cite{MW}. It is easily checked that if $T$ is a contraction on
$\mathcal H$, then $\tilde T(x,y):=(Tx,Ty)$ is a contraction on $\tilde H$, and we deduce
the result from the complex case.
\end{proof}

{\bf Definition.} A linear operator $T$  on a Banach space $E$ is called {\it polynomially 
bounded} if there exists $K$ such that
$\|P(T)\| \le K \sup_{|\lambda|\le 1} |P(\lambda)|$ for every polynomial $P(t)$ with
 scalar coefficients.

Polynomially bounded operators are power-bounded, but the converse is false. Solving a
long standing problem, Pisier \cite{Pi} constructed a polynomially bounded operator on a
Hilbert space $\mathcal H$ which is not similar to a contraction (note that contractions
on $\mathcal H$ are polynomially bounded, by von Neumann's inequality).
When $T$ is invertible on $\mathcal H$ with $\sup_{n\in \mathbb Z}\|T^n\|< \infty$,
it is similar to a unitary operator \cite{SN0}.

Cohen \cite{Co} constructed an invertible $T$ on $L^p$, $p\ne 2$, such that
$T$ is polynomially bounded, so $\sup_{n\in \mathbb Z}\|T^n\|< \infty$, but $T$ is not
similar to an invertible isometry. Zarrabi \cite{Z} studied conditions for a contraction
on a Banach space to be polynomially bounded.
\smallskip

{\bf Example.} {\it A non-WAP polynomially bounded operator.}

\noindent
Let $\mathbb D$ be the unit disk and $X=\overline{\mathbb D}$. On $C_\mathbb C(X)$ define
$Tf(z):= zf(z)$. Then $\|P(T)f(z)\|_\infty=\max_{z\in X} |P(z)f(z)| \le \|P\|_\infty\|f\|_\infty$,
so $T$ is polynomially bounded.
 However, $T$ is not mean ergodic, since for $f \equiv 1$ the averages
 $\frac1n\sum_{k=1}^n T^kf$ do not converge  to a continuous function; hence $T$ is not WAP.

The similar definition $Tf(z):=zf(z)$ for $f \in C_\mathbb C(\mathbb T)$ yields an invertible
isometry which is polynomially bounded and not mean ergodic, therefore not WAP.

\begin{prop} \label{poly-bounded}
Let $T$ be polynomially bounded on a Banach space $E$. Then
$$\Big\| \frac1N \sum_{n=1}^N \mu(n) T^n\Big\| \to 0.$$
\end{prop}
\begin{proof}
By the maximum modulus principle,
$$\|P(T)\| \le K\sup_{|\lambda|\le 1} |P(\lambda)|= K\sup_{|\lambda|=1} |P(\lambda)|.$$
The result follows by applying \eqref{davenport} to 
$P(\lambda)=\frac1N\sum_{n=1}^N \mu(n)\lambda^n$, which also gives the rate of convergence.
\end{proof}

{\bf Remark.} Propositions \ref{H-contraction} and \ref{poly-bounded} show a common
rate of convergence for all polynomially bounded operators (which include Hilbert space
contractions). Algom and Wang \cite{AW} showed that for every rate function
$r_n \downarrow 0$ there is a topological dynamical system $(X,\tau)$ with $h_{top}(\tau)=0$
which satisfies \eqref{sarnak-op}, but for some $g \in C(X)$ and $x\in X$,
$\limsup \frac{\frac1N\sum_{n=1}^N \mu(n) g(\tau^n x)}{r_N}>0$; this means that
 convergence in \eqref{sarnak-op1} can be arbitrarily slow.

\begin{prop} \label{wap-T}
 Let $T$ be power-bounded on $E$ and define
$WAP(T):=\{v \in E: (T^nv) \text{ is weakly sequentially compact}\}$. Then $WAP(T)$ is
a $T$-invariant closed subspace of $E$, and $T_{|WAP(T)}$ is weakly almost periodic.
\end{prop}
\begin{proof} We assume $WAP(T) \ne \{0\}$, and also that $\|T\|\le 1$. It is easy to see
that $WAP(T)$ is a linear manifold.  Let $(v_k) \subset WAP(T)$ converge in norm to $v$.
Fix a subsequence $(n_j)$.  By the diagonal process, there is a subsequence
$(r_\ell) \subset (n_j)$ such that $T^{r_\ell}v_k$ converges weakly to $w_k$ as
$\ell\to\infty$, for each $k$. For $\phi \in E^*$ we have
$$
|\phi(w_k)-\phi(w_p)| =
\lim_{\ell\to\infty} |\phi(T^{r_\ell}(v_k-v_p))| \le \|\phi\|\cdot\|v_k-v_p\|.
$$
This implies $\|w_k-w_p\| \le \|v_k-v_p\|$; hence $w_k$ converges, say to $w$.
For $\phi \in E^*$  we then have
$$
\limsup_{\ell\to\infty}|\phi(T^{r_\ell}v)-\phi(w)| \le \|\phi\|(\|v-v_k\|+\|w_k-w\|).
$$
This yields that $T^{r_\ell}$ converges weakly to $w$. Since this holds for any
subsequence $(n_j)$, we conclude that $v \in WAP(T)$. Thus $WAP(T)$ is closed;
its invariance under $T$ is obvious
\end{proof}

\begin{theorem} \label{dor}
Let $E$ be a complex Banach space which does not contain an isomorphic copy of $\ell_1$.
Let $T$ be a power-bounded operator on $E$ such that
$$
\Big\|\frac1N\sum_{n=1}^N \lambda^n T^nv\Big\| \to 0 \quad
\forall v\in E,\  \forall  \ |\lambda|=1.
$$
Then
\begin{equation} \label{strong-all}
\Big\|\frac1N\sum_{n=1}^N b_n T^nv\Big\| \to 0 \quad
 \text{ for every } v\in E \text{ and } (b_n)_{n\ge 1} \in \ell_\infty.
\end{equation}
In particular, every $v \in E$ satisfies \eqref{strong}.
\end{theorem}
\begin{proof}
By \cite[Corollary 5 and Theorem 2]{JL2}, the assumption is equivalent to
\begin{equation} \label{weak-mix}
\frac1N \sum_{n=1}^N |\phi(T^nv)| \to 0\quad  \forall v\in E,\ \forall \phi \in E^*.
\end{equation}
Fix $v \in  E$. By Jones and Lin \cite{JL1} for the convergence, for $(b_n) \in\ell_\infty$ we obtain
$$
\Big\| \frac1N \sum_{n=1}^N b_n T^nv \Big\| =
\sup_{\|\phi\|_{E^*}\le 1} \Big| \frac1N \sum_{n=1}^N b_n \phi(T^nv) \Big|\le
$$
$$
 \|(b_n)\|_\infty \sup_{\|\phi\|_{E^*}\le 1}\frac1N \sum_{n=1}^N \Big|  \phi(T^nv) \Big| \to 0.
$$
 Note: The assumption on $T$  means that $T^*$ has no unimodular eigenvalues
\cite[Theorem 1.3, p. 73]{K}, and it yields \eqref{weak-mix} by the assumption on $E$.
\end{proof}

{\bf Remark.} The condition on $E$ is satisfied when $E^*$ is separable; a simple example
is $c_0$. An example with $E^{**}$ separable is the (separable) James space \cite{Ja} $J$,
with $\dim(J^{**}/J)=1$.
\smallskip

The next result shows that \eqref{weak-mix}, which by the end of the proof of Theorem \ref{dor}
 yields \eqref{strong-all},  does not imply weak almost periodicity.

\begin{prop} \label{not-wap}
There exists an invertible isometry $T$ on a separable Banach space $E$ such that
\eqref{weak-mix} and \eqref{strong-all} hold, in particular $T$ is mean ergodic,
 but $T$ is not weakly almost periodic.
\end{prop}
\begin{proof}
We first prove the real case. By Theorem 6 of Jones and Lin \cite{JL2}, there exists an invertible isometry $T$
on a separable  real Banach space $E$, such that for some $u \in E$, for no subsequence $(n_j)$ does $T^{n_j}u$
converge to 0 weakly, but $\frac1N\sum_{n=1}^N |\psi(T^nu)| \to 0$ for every $\psi \in E^*$.

By the construction in \cite{JL2}, $E:= \text{closed lin. span }\{T^ku: k \in \mathbb Z\}$.
We first prove \eqref{weak-mix}.

Fix $k \ge 0$. For $ \psi \in E^*$ we have, for $N>k$,
$$
\frac1N \sum_{n=1}^N |\psi(T^{n+k}u)| = \frac1N \sum_{n=k+1}^{N+k} |\psi(T^{n}u)| =
$$
$$
 \frac1N \sum_{n=1}^{N} |\psi(T^{n}u)| +  \frac1N \sum_{n=N+1}^{N+k} |\psi(T^{n}u)| -
 \frac1N \sum_{n=1}^{k} |\psi(T^{n}u)| \to 0.
$$
Hence
$T^ku \in E_0:=\{v\in E: \frac1N \sum_{n=1}^{N} |\psi(T^{n}v)| \to 0 \ \forall \psi \in E^*\}$.
Similarly, $T^ku\in E_0$ when $k<0$. Thus $E_0$ contains the linear combinations of
$(T^ku)_{k\in \mathbb Z}$, which are dense in $E$. We prove $E_0=E$.

Let $v \in E$. For $\varepsilon >0$ there exists $w \in E_0$ with $\|v-w\|<\varepsilon$;
then for $\psi \in E^*$ we obtain
$\limsup_N \frac1N \sum_{n=1}^{N} |\psi(T^{n}v)| \le \|\psi\|\varepsilon$.
Letting $\varepsilon \to 0$ we conclude that $v \in E_0$. Hence \eqref{weak-mix} holds.

Since for no subsequence does $(T^{n_j}u)$ converge weakly to 0 ,
but 0 is a weak cluster point of $(T^ku)$, this sequence is not weakly sequentially compact,
so $T$ is not weakly almost periodic. Thus the proposition  is proved for the real case.

An inspection of the proof of \cite[Theorem 6]{JL2} shows that if we start with the vector
space of  bilateral complex sequences with finitely many non-zero terms, then we obtain
a separable complex Banach space with an isometry $T$ with the same properties as stated
in that theorem. Then the above arguments prove our proposition in the complex case.
\end{proof}

\begin{theorem} \label{dor2}
Let $E$ be a complex Banach space which does not contain an isomorphic copy of $\ell_1$.
Let $T$ be a power-bounded operator on $E$ such that $T^*$ has finitely many unimodular
eigenvalues $\lambda_1, \dots,\lambda_d$, and assume that $\lambda_j^{-1}T$ is mean
ergodic for $1\le j\le d$. Then
$$
\Big\|\frac1N\sum_{n=1}^N \mu(n)T^nv\Big\| \to 0 \quad \text{ for every } v\in E .
$$
\end{theorem}
\begin{proof}
In the proof of \cite[Theorem 2.4]{CL} it was proved that if $T_1,\dots, T_d$ are
commuting mean ergodic power-bounded operators,  then\footnote{The proof depends on
\cite[Lemma 2.3]{CL}, where it is shown, in the notations there, that the projection
$P:=\prod_{j=1}^d P_j$ annihilates $X_0:=\overline{\sum_{j=1}^d(I-T_j)X}$.
To show $E_0= \ker P$ as claimed there, note that by Hahn-Banach separation,
if $z \in\ker P$ is not in $E_0$, we get $\phi\in X^*$ such that $T_j^*\phi=\phi$
for every $j$, so $P^*\phi=\phi$, and then $0\ne \phi(z)=P^*\phi(z)= \phi(Pz)=0$,
a contradiction.}
$$
E=\big[\bigcap_{j=1}^d \overline{(I-T_j)E}\big] \oplus \overline{\sum_{1\le j\le d} F(T_j)},
$$
where $F(T):=\{v: Tv=v\}$. Applying this decomposition  to $T_j=\lambda_j^{-1}T$ we obtain
$E=E_0\oplus E_1$, with $E_0:=\bigcap_{j=1}^d \overline{(\lambda_j I-T)E}$ and
$$
E_1:= \overline{\sum_{1\le j\le d} F(T_j)}=
\text {closed lin. span } \{v: Tv=\lambda_j v,\ 1\le j \le d\}.
$$
For $\lambda \in \mathbb T$ not an eigenvalue of $T^*$ we have
$\|\frac1N \sum_{n=1}^N \bar \lambda^nT^n v\|\to 0$ for every $v \in E$. For each
$\lambda_j$ we have $ \|\frac1N \sum_{n=1}^N \bar \lambda_j^n T^n v\|\to 0$ for $v\in E_0$.
Applying Theorem \ref{dor} to the restriction of $T$ to $E_0$ we have
$$
\Big\|\frac1N\sum_{n=1}^N b_n T^nv\Big\| \to 0 \quad
\text{ for every } v\in E_0 \text{ and } (b_n)_{n\ge 1} \in \ell_\infty.
$$
To the generators of $E_1$ we apply Davenport's inequality, and obtain
$$\Big\|\frac1N\sum_{n=1}^N \mu(n)T^nv\Big\| \to 0 \quad \text{ for every } v\in E_1 ,$$
This completes the proof of the theorem.
\end{proof}

{\bf Remark.} In some of our theorems, when $E$ is over $\C$, $\ T$ is {\it rotationally mean
ergodic}, meaning that all the operators $\lambda T$, $|\lambda|=1$, are mean ergodic.

\begin{prop} \label{quasi-compact}
Let $E$ be a Banach space which does not contain an isomorphic copy of $\ell_1$.
Let $T$ be a power-bounded quasi-compact operator on $E$, i.e.  for some
$n\ge 1$ there exists a compact operator  $Q$ such that $\|T^n-Q\|<1$. Then
$$
\Big\|\frac1N\sum_{n=1}^N \mu(n)T^nv\Big\| \to 0 \quad \text{ for every } v\in E .
$$
\end{prop}
\begin{proof} We first prove the proposition for $E$ over $\mathbb C$. We observe that
for $|\lambda|=1$ we have $\|(\lambda T)^n -\lambda^nQ\|<1$, so $\lambda T$ is
quasi-compact.
	
By \cite[Theorem VIII.8.3 and Corollary VIII.8.4]{DS},  $T$ is uniformly mean ergodic, and
$\sigma(T)\cap \mathbb T$ is finite. Hence also $\sigma(T^*)\cap \mathbb T$ is finite,
so in particular $T^*$ has only finitely many unimodular eigenvalues $\lambda_1,\dots\,\lambda_m$.
Since all $\lambda_j^{-1} T$ are quasi-compact, they are (uniformly) mean ergodic.
Now the claim of our proposition follows from Theorem \ref{dor2}.
\smallskip

We now prove the proposition when $E$ is over $\mathbb R$. Let $\tilde E$ be the
complexification of $E$ with the Taylor norm, as described in the proof of Theorem
\ref{wap}, and $\tilde T(u,v)=(Tu,Tv)$ for $(u,v) \in\tilde E$. Then $\tilde T$ is 
power-bounded, and by \cite[Proposition 4]{ MST} $\|\tilde T^n -\tilde Q\|<1$. It is easy 
to show that $\tilde Q(u,v)=(Qu,Qv)$ is compact on $\tilde E$, so $\tilde T$ is quasi-compact,
and the first part of the proof applies to $\tilde T$. Hence for $v \in E$ we have
$$
\Big\|\frac1N\sum_{n=1}^N \mu(n)T^nv\Big\| =
\Big\|\frac1N\sum_{n=1}^N \mu(n)\tilde T^n(v,0) \Big\| \to 0
$$
\end{proof}

{\bf Remark.} Quasi-compactness is used in the study of Markov operators. See the notes in
\cite[p. 730]{DS}.
\medskip

\begin{lemma} \label{ell_1}
Let $E$ be a complex Banach space which does not contain an isomorphic copy of the complex $\ell_1$. Then $E$ as a real Banach space does not contain an isomorphic
copy of the real $\ell_1$.
\end{lemma}
\begin{proof} assume that $E$ contains an isomorphic copy of the real  $\ell_1$, with
$(v_j)_{j\ge 1}$ equivalent to the basis of the real $\ell_1$. By the last paragraph
of Roesenthal's \cite{Ro}, if $(v_j)$ does not contain a subsequence equivalent to
 the basis of the complex $\ell_1$, then there is another sequence in $E$ which is
 equivalent to  the basis of the complex $\ell_1$. In either case, this contradicts the assumption on $E$; hence $E$ does not contain an isomorphic copy of the real $\ell_1$.
\end{proof}

\begin{theorem} \label{rosenthal-rep}
Let $E$ be a (real or complex) Banach space which does not contain an isomorphic
copy of $\ell_1$. Let $T$ be an invertible operator on $E$ with
$\sup_{n\in \mathbb Z}\|T^n\| < \infty$. Then  \eqref{sarnak-op-E} holds.
\end{theorem}
\begin{proof} Lemma \ref{ell_1} shows that it is enough to prove for $E$  real.
We may assume that $T$ is an invertible isometry of $E$, since it is so
under the equivalent norm $\||v\||:= \sup_{n\in \mathbb Z}\|T^nv\|$. Let $X$ be the unit ball
of $E^*$, which is compact in the weak* topology, and define $\tau: X\longrightarrow X$
by $\tau \phi=T^*\phi$, $\phi\in X$.

{\it Claim: The system $(X,\tau)$ is tame (in the sense of \cite{G}).}

\noindent
We show that $(X,\tau)$ is faithfully represented in $E$. The Abelian group $G=\mathbb Z$
 is clearly  represented in the group  $Iso(E)$ of invertible linear isometries of $E$, by
$h(n):= T^n,\  n\in\mathbb Z$. Let $\alpha: X \to E^*$ be the identity embedding.
By the definition of $\tau$ we have $\tau^n\alpha(\phi)(v)= \phi(T^nv)=\phi(h(n)v)$.
Obviously, the pair $(h,\alpha)$ is a faithful representation in $E$, by the definition in
\cite[Definition 3.1]{GMe}. Since $E$ does not contain an isomorphic copy of $\ell_1$
($E$ is a Rosenthal space in the terminology of \cite{GMe}), Theorem 6.1 of Glasner and
Megrelishvili \cite{GMe} yields that $(X,\tau)$ is tame.
\smallskip

By \cite[Corollary 1.5]{HWY}, the tame system $(X,\tau)$ satisfies \eqref{sarnak-op}.
For $v \in E$ define $f_v \in C(X)$ by $f_v(\phi)=\phi(v)$, $\phi\in X$. Then
$$
\phi\Big(\frac1N\sum_{n=1}^N\mu(n)T^n v\Big)=
\frac1N\sum_{n=1}^N\mu(n) f_v(\tau^n \phi) \to 0, \qquad \forall \phi \in X.
$$
This yields the theorem.
\end{proof}

{\bf Remark.} The example following Corollary \ref{stationary} shows that without
the condition that $E$ does not contain $\ell_1$, Theorem \ref{rosenthal-rep} may fail.
\medskip

{\bf Definition.} An operator $T$ on a Banach space $E$ is called {\it (weakly)  rigid} if
 for some subsequence $(n_k)$ we have $T^{n_k} v \to v$ (weakly) for every $v\in E$.

If $T$ is a rigid power-bounded operator, then $T$ is invertible, and $T^{-1}$ is also
power-bounded and rigid \cite[Proposition 3.2]{CMP}. If $T$ is a rigid contraction, it is
an (invertible) isometry, since $\|v\|=\lim_k \|T^{n_k}v\|=\lim\|T^n v\|\le \|Tv\| \le \|v\|$.

\begin{cor} \label{rigid-sarnak}
Let $E$ be a Banach space which does not contain an isomorphic copy of $\ell_1$, and
let $T$ be a weakly rigid power-bounded operator on $E$. Then  \eqref{sarnak-op-E} holds.
\end{cor}
\begin{proof}
By Proposition \ref{rigid-isometry} $T$ is invertible and $T^{-1}$ is power-bounded,
so Theorem \ref{rosenthal-rep} applies.
\end{proof}

{\bf Remark.} A doubly power-bounded invertible $T$ (i.e.
$\sup_{n\in\mathbb Z}\|T^n\|=M<\infty$) on an infinite-dimensional Banach space is
 never quasi-compact. Indeed, if $T$ is quasi-compact, then, by \cite[Lemma 2.2.4, p. 88]{K},
there exists a sequence of compact operators $(Q_n)_{n\ge 1}$ such that $\|T^n-Q_n\| \to 0$,
and then $\|I-T^{-n}Q_n\| \le M\|T^n-Q_n\| \to 0$, so for large $n$ the compact operator
$T^{-n}Q_n$ is invertible, which is impossible.

Consequently, for $T$ power-bounded,  quasi-compactness and weak rigidity are mutually
exclusive.

\begin{prop} \label{not-UE}
There exists a minimal topological system $(X,\tau)$ which is not uniquely ergodic,  such that
\eqref{sarnak-op} holds for every $f \in C(X)$.
\end{prop}
\begin{proof}
By Furstenberg's construction \cite{F} (cf. also \cite[p. 85]{P}), there exist an irrational
$\alpha$ and $h$ a smooth periodic function of period $1$  such that the map of the torus
defined by $(x,y) \mapsto (x+\alpha,y+h(x)) \mod 1\,$ is minimal and not uniquely ergodic.
By Liu and Sarnak  \cite[Theorem 1.2]{LS}, with some additional  assumptions on $h$, this
map satisfies \eqref{sarnak-op}.
\end{proof}

\begin{cor}
There exists a contraction $T$ on a Banach space $E$ which satisfies \eqref{sarnak-op-E}
and is not mean ergodic.
\end{cor}
\begin{proof}
Let $(X,\tau)$ be the system of Proposition \ref{not-UE} and $E=C(X)$ with $Tf=f\circ\tau$ 
the induced operator on $C(X)$. Then $T$ satisfies \eqref{sarnak-op-E}.
Since $\tau$ is minimal and not uniquely ergodic, the induced operator $T$ is not mean ergodic
on $C(X)$ \cite[p. 180]{K}.
\end{proof}

\medskip

\section{Dynamical topological entropy for power-bounded operators}

In this section we define topological entropy for a power-bounded operator on a Banach space.
We show that our definition is well-adapted to deal with Sarnak's conjecture.
\medskip

The topological entropy $h_{top}(\tau)$ of a compact Hausdorff dynamical system
$(X,\tau)$, with $\tau$ a continuous self map of $X$, was defined by Adler et al.
\cite{AKM}, who proved  its basic properties. A corollary of \cite[Theorem 4]{AKM}
is that if $Y \subset X$ is closed and $\tau$-invariant, then
$h_{top}(\tau_{|Y}) \le h_{top}(\tau)$. The proof of \cite[Theorem 5]{AKM} shows that
if $\pi$ is a continuous factor map of $(X,\tau)$ onto some $(Y,\varphi)$ (i.e.
$\pi\circ\tau=\varphi\circ \pi$), then $h_{top}(\varphi) \le h_{top}(\tau)$.

Goodman \cite{Go} proved the following {\it variational principle} (conjectured in
\cite{AKM}):
\begin{equation} \label{goodman}
h_{top}(\tau)=
\sup\{h_\nu(\tau): \nu\circ\tau^{-1}=\nu, \quad\nu \text{ a probability on } X\},
\end{equation}
where $h_\nu(\tau)$ is the Kolmogorov-Sinai entropy of the probability preserving system
$(X,\nu,\tau)$.
\smallskip

When $X$ is compact metric, there are two other definitions of topological entropy of $\tau$, 
which are both equal to $h_{top}(\tau)$ as defined in \cite{AKM}; see \cite[pp. 160-163]{Do}.


\medskip

\subsection{Topological entropy for contractions using the dual operator.}

One way of defining the entropy of a contraction $T$ on a Banach space $E$ (every
power-bounded $T$ is a contraction in an equivalent norm) is the following.
Let $B$ be the unit ball in $E^*$. Then $B$ is a compact Hausdorff space in the
weak* topology, which is metrizable when $E$ is separable. Since $T$ is a contraction,
$T^*$ maps $B$ into itself, and is weak* continuous as a dual operator. We now
define $h^*_{top}(T):=h_{top}(T^*{|B})$, the topological entropy of the system
$(B,T^*{|B})$.

\begin{prop} \label{subspace}
Let $T$ be a contraction on a Banach space $E$, and let $E_0 \subset E$ be
a $T$-invariant closed subspace. The $h^*_{top}(T_{|E_0})\le h^*_{top}(T)$.
\end{prop}
\begin{proof}
It is known that $E_0^*=E^*/E_0^\perp$ (e.g \cite[Exercise 18(a), p. 72]{DS}).
Denote $T_0=T|E_0$, and let $B_0$ be the unit ball of $E_0^*$.
For $\phi \in E^*$ let $\pi(\phi) \in E_0^*$ be defined by $\pi(\phi)(w)= \phi(w)$,
$w \in E_0$. Obviously $\|\pi(\phi)\| \le \|\phi\|$, so $\pi$ maps $B$ into $B_0$.
By the Hahn-Banach theorem, for $\psi \in B_0$ there is $\phi \in E$ with
$\phi(w)=\psi(w)$ for every $w \in E_0$ and $\|phi\|=\|\psi\|$. Thus $\pi$ maps
$B$ onto $B_0$.
	
For $\phi \in B$ and $w \in E_0$ we have
$$
(T_0^*(\pi(\phi))w=\pi(\phi)(T_0w)=\phi(T_0w)=(T^*\phi)w = (\pi(T^*\phi))w.
$$
Hence $(T_0^*(\pi(\phi)) = \pi(T^*\phi)$, for every $\phi \in B$, so $\pi$ is
a factor map of $(B,T^*|B)$ onto $(B_0,T^*_0|B_0)$. We prove continuity of $\pi$
in the weak* topologies, using nets. If $\phi_\alpha \to \phi$ weak*, then
$(\pi(\phi_\alpha))w =\phi_\alpha(w) \to \phi(w)= =(\pi(\phi))w$ for every
$w \in E_0$. Thus $\pi(\phi_\alpha) \to \pi(\phi)$ in the weak* topology of $B_0$,
to $\pi$ is a continuous factor map. By \cite{AKM}, $h_{top}(T_0^*|B_0)\le h_{top}(T^*|B)$,
	which is $h^*_{top}(T_0) \le h^*_{top}(T)$ by definition.
\end{proof}
	
\begin{prop} \label{h*}
Let $T$ be a contraction on a Banach space $E$. If $h_{top}^*(T)=0$
and Sarnak's conjecture holds, then \eqref{sarnak-op-E} holds.
Moreover, every power $T^n$ satisfies \eqref{sarnak-op-E}.
\end{prop}
\begin{proof}
For $v \in E$ the function $f_v(\phi):=\phi(v), \ \phi \in B$, is continuous on $B$
with its weak* topology, so  Sarnak's conjecture applied to $(B,T^*_{\ |B})$ yields
\begin{equation} \label{sarnak*}
\frac1N\sum_{n=1}^N \mu(n) \phi(T^n v) =
\frac1N\sum_{n=1}^N \mu(n) f_v(T^{*n}\phi) \to 0, \quad \forall  \phi\in B,
\end{equation}
which yields \eqref{sarnak-op-E}. In fact, In view of \cite[Corollary 10]{AKLR},
\eqref{sarnak-norm} yields in Proposition \ref{h*} that
$ \big\| \frac1N\sum_{n=1}^N \mu(n) T^n v\big\| \to 0 \quad \forall v \in E $.

Fix $n >1$. Then $h_{top}^*(T^n)= n\cdot h_{top}^*(T)=0$ by \cite[Theorem 2]{AKM},
so $T^n$ satisfies  \eqref{sarnak-op-E} by the first part of the proof.
\end{proof}

{\bf Remark.} In Propositon \ref{h*} it is enough to assume that Sarnak's conjecture
holds for continuous maps of compact {\it metric} spaces. Indeed, by Proposition
\ref{subspace}  $h^*_{top}(T)=0$ implies that for fixed $v$, also $h^*_{top}(T_{|E_v})=0$,
where $E_v$ is the closed subspace genereated by $(T^nv)_{n\ge0}$, which is separable;
then the weak* topology of the  unit ball in $E_v^*$ is metrizable. This also shows that
when $E$ is not separable, the assumption $h^*_{top}(T)=0$ can be weakened to {\it for
every $T$-invariant separable subapace $E_0\subset E$ we have $h^*_{top}(T_{|E_0})=0$.}

\begin{prop} \label{bernoulli}
Define $T$ on $\ell^1(\mathbb Z)$ by $T(f_k)=(f_{k-1})$. Then $h^*_{top}(T) >0$.
\end{prop}
\begin{proof} The dual $T^*$ is the left shift $S(g_k)=(g_{k+1})$ on
$\ell^\infty(\mathbb Z)$.  Let $B$ be the unit ball of $\ell^\infty(\mathbb Z)$ 
and denote by  $S_B$ the restriction of $S$ to $B$. Then $h^*_{top}(T)=h_{top}(B,S_B)$. 
	Let $Y:=\{-1,1\}^\mathbb Z \subset B$. Then $Y$ is $S$-invariant,
and the product measure $\nu^\mathbb Z$, with $\nu(1)=\nu(-1)=\frac12$ is invariant
for $S_{|Y}$; it is known that $h_{\nu^\mathbb Z}(S_{|Y})> 0$. Hence
$h_{top}(B,S_B) \ge h_{top}(Y,S_{|Y})>0$.
\end{proof}

{\bf Remark.} In the example following Corollary \ref{stationary}, $T$ fails
\eqref{sarnak-op-E}; however, it is not a counter-example to Sarnak's conjecture,
since $h^*_{top}(T)=h_{top}(B,S_B)\ne 0$. The latter inequality is probably well-known, 
but we have found no reference.

\begin{prop}  \label{tau-on-X}
Let $\tau$ be a continuous map of a compact metric space $X$ to itself
and let $Tf:=f\circ \tau$ be the induced contraction on $C(X)$.
Then $h^*_{top}(T) \ge h_{top}(\tau)$.
\end{prop}
\begin{proof} The map $\varphi(x)=\delta_x$ is a continuous map of $X$ into  $B$,
the unit ball of $C(X)^*$, with its weak* topology. Then $T^*\delta_x =\delta_{\tau x}$,
so $\varphi(X)$ is a compact $T^*$-invariant set in $B$. By the homeomorphism of
$X$ and $\varphi(X)$ we have $h_{top}(\tau)= h_{top}(T^*_{|\varphi(X)})\le
h_{top}(T^*_{|B})=h^*_{top}(T)$.
\end{proof}


\begin{prop} \label{homeo}
Let $\tau$ be a homeomorphism of a compact Hausdorff space and $Tf:=f\circ\tau$ the 
induced operator on $C(X)$. Then $h^*_{top}(T)=0$ if and only if $h_{top}(\tau)=0$.
\end{prop}
\begin{proof} This is a combination of Proposition 3.1 and Theorem 3.5 of Kerr and Li 
\cite{KL}.
\end{proof}

Let $X$ be a compact metric space. Then $C(X)$ is separable, so $\mathcal M_1(X)$,
the unit ball of $C(X)^*$, is weak* compact metric; it consists of signed
measures $\nu$ with $\|\nu\| \le 1$. A continuous map $\tau: X \mapsto X$ induces a 
continuous map $\tau^*$ on $\mathcal M_1(X)$, defined by $\tau^*\nu:=\nu\tau^{-1}$.
If $Tf=f\circ\tau$ is the induced operator on $C(X)$, then $\tau^*$ is the restriction  
of $T^*$ to $\mathcal M_1(X)$.

\begin{theorem} \label{metric}
Let $(X,\tau)$ be a topological system with $X$  metric, and let $Tf:=f\circ\tau$ be the 
induced operator on $C(X)$. Then $h^*_{top}(T)=0$ if and only if $h_{top}(\tau)=0$.
\end{theorem}
\begin{proof} If $\tau$ is invertible, this is Proposition \ref{homeo}. We therefore assume
that $\tau$ is not invertible. By Proposition \ref{tau-on-X} $h^*_{top}(T) \ge h_{top}(\tau)$,
so we have to prove only that $ h_{top}(\tau)=0$ implies $h^*_{top}(T)=0$. We thus 
assume $ h_{top}(\tau)=0$.

Let $(\hat X,\hat\tau)$ be the extension of $(X,\tau)$ given by Proposition \ref{extension},
and let $\hat T$ be the induced operator on $C(\hat X)$.
By Proposition \ref{both0} $ h_{top}(\tau)=0$ implies $h_{top}(\hat\tau)=0$, and by 
Proposition \ref{homeo} $h_{top}(\mathcal M_1(\hat X),\hat\tau^*)= h^*_{top}(\hat T)=0$.
 The factor map $\pi$ of $(\hat X,\hat\tau)$ onto $(X,\tau)$ induces a map $ \pi^*$ 
from $\mathcal M_1(\hat X)$ into $\mathcal M_1(X)$, defined by $\pi^*\eta=\eta\pi^{-1}$. 
Given $\nu \in\mathcal M_1(X)$, we have $\pi^*\hat\nu=\nu$ by the construction of 
the natural extension (for $\nu^+$ and $\nu^-$ separately), so $\pi^*$ is onto
$\mathcal M_1(X)$. The factor map satisfies $\tau\circ \pi=\pi\circ \hat\tau$, which yields
$$
\tau^*(\pi^* \eta)= \eta\pi^{-1}\tau^{-1}=
\eta\hat\tau^{-1}\pi^{-1} =\pi^*(\hat\tau^*\eta), \qquad \forall\eta\in\mathcal M_1(\hat X).
$$
Thus $\tau^*\pi^*=\pi^* \hat\tau^*$, so $\pi^*$ is a factor map of 
$\mathcal M_1(\hat X)$ onto $\mathcal M_1(X)$; hence
$h^*_{top}(T)=h_{top}(\mathcal M_1(X),\tau^*) \le 
h_{top}(\mathcal M_1(\hat X),\hat\tau^*) =0$.
\end{proof}

{\bf Remark.} Theorem \ref{metric} implies that \cite[Theorem A]{GW} of  Glasner and Weiss
holds also for $(X,\tau)$ with non-invertible $\tau$.

\begin{theorem} \label{sarnak-equiv}
The following conditions are equivalent:

(i) Every system $(X,\tau)$ of a compact metric space with $h_{top}(\tau)=0$
satisfies \eqref{sarnak-op1}.

(ii) Every contraction $T$ on a separable Banach space $E$ with $h^*_{top}(T)=0$
satisfies \eqref{sarnak-op-E}.

(iii) Every contraction $T$ on a  Banach space $E$ with $h^*_{top}(T)=0$
satisfies \eqref{sarnak-op-E}.
\end{theorem}
\begin{proof} Assume (i). Let $T$ be a contraction of a separable $E$, with $h^*_{top}(T)=0$.
Then $B$, the unit ball of $E^*$, is compact metric in  the weak* topology,
and $\tau:=T^*_{|B}$ is a continuous self-map of $B$ with
 $h_{top}(\tau)=0$, by definition of $h^*_{top}(T)$. Then
\eqref{sarnak-op-E} follows from (i), as in the proof of Proposition \ref{h*}.
\smallskip

Assume (ii). Let $(X,\tau)$ be a topological system with $X$ compact metric, satisfying
 $h_{top}(\tau)=0$, and define $Tf =f\circ \tau$ for $f \in C(X)$. 
By Theorem \ref{metric}  $h^*_{top}(T)=0$, so by (ii), since $E:=C(X)$ is separable,
$T$ satisfies \eqref{sarnak-op-E}, which in this case is  \eqref{sarnak-op}.
\smallskip

Obviously (iii) implies (ii). Assume (ii). Let $E$ be a non-separable Banach space and $T$
a contraction with $h^*_{top}(T)=0$. Given $v \in E$, let $E_v$ be the closed linear manifold 
generated by $(T^n v)_{n\ge 0}$. Then $E_v$ is a separable $T$-invariant space. 
By Proposition \ref{subspace}, $h^*_{top}(T_{|E_v})=0 \le h^*_{top}(T)=0$, so by
(ii) $\frac1N\sum_{n=1}^N\mu(n)T^nv \to 0$ weakly. Since this holds for any $v\in E$,
\eqref{sarnak-op-E} is satisfied.
\end{proof}

{\bf Remark.} Theorem \ref{sarnak-equiv} yields an operator formulation equivalent to
Sarnak's conjecture for metric spaces.

\begin{prop} \label{wap-zero}
Let $T$ be a weakly almost periodic contraction on a separable  Banach space $E$.
Then $h^*_{top}(T)=0$.
\end{prop}
\begin{proof}
Let $B$ be the unit ball of $E^*$, and for $f \in C(B)$ put $Sf(\phi)=f(T^*\phi)$.
We first show that $S$ is WAP on $C(B)$. Define $A=WAP(S)$, where
$$
WAP(S)=\{f\in C(B): (S^n f)_{n\ge 0}\ \text{\rm is weakly sequentially compact}\}.
$$
Then $A$ is a closed $S$-invariant subspace of $C(B)$, by Proposition \ref{wap-T}.
For $v \in E$, the function $f_v(\phi):=\phi(v)$ is in $A$ since $T$ is WAP.
Since $f \in A$ if and only  if every sequence $(n_j)$ has a subsequence  $(k_\ell)$ with
$S^{k_\ell}f$ converging pointwise to a continuous function, it follows easily
that $A$ is closed under products. Constants are in $A$, and the functions
$\{f_v: v\in E\} \subset A$ separate the points of $B$. When $E$ is over $\mathbb C$, $A$
is closed under conjugation; hence by the Stone-Weierstrass theorem, $WAP(S)=A=C(B)$.

By a result of Bourgain, Fremlin and Talagrand, see \cite[Theorem 3.11]{vD},
if $(X,\tau)$ is a WAP system, then for any $f \in C(X)$, the sequence
$(f\circ\tau^n)_{n\ge 0}$ does not contain an $\ell_1$ sequence, i.e. the system is
	{\it tame} in the terminology of \cite{G}.

Since $S$ is WAP on $C(B)$, it is mean ergodic, i.e. $\frac1N\sum_{k=1}^N S^kf$ converges
in norm for any $f \in C(B)$. It follows that the support $\sigma_\nu$ of any ergodic invariant
probability $\nu=S^*\nu$ is a minimal invariant (absorbing) set \cite[Corollary 2.3]{S}.
Obviously $S_{\nu}:= S_{|\sigma_\nu}$ is WAP on $C(\sigma_\nu)$, so
$(\sigma_\nu, T^*_{|\sigma_\nu})$ is a weakly almost periodic minimal topological dynamical
system, hence a minimal tame system. By Glasner's \cite[Theorem 1.10]{G}\footnote{Glasner
told us that the result holds also without invertibility}, $h_{top}(T^*_{|\sigma_\nu})=0$,
hence $h_\nu(T^*_{|\sigma_\nu})=0$. Using the ergodic decomposition of a general $S$-invariant
probability $\nu$, the integral representation given in \cite[p. 78, Theorem 13.3(b)]{DGS}
yields $h_\nu(T^*_{|B})=0$, and the variational principle yields
$h^*_{top}(T)=h_{top}(T^*_{|B})=0$.
\end{proof}

\begin{cor} \label{compact}
Let $T$ be a contraction on a separable Banach space $E$. If some power $T^n$ is
weakly compact, then $h^*_{top}(T)=0$, and \eqref{strong} holds.
\end{cor}
\begin{proof} Remember that $T^n$ is weakly compact if it maps bounded sets to
weakly sequentially  compact sets. Fix $v \in E$. Given a subsequence $(n_i)$, 
for some $0\le j<n$ there are infinitely many $n_i\equiv j \mod n$, so there is a 
subsequence $(k_\ell=r_\ell n+j)$ such that $T^{k_\ell}v =T^{r_\ell n}(T^jv)$
converges weakly. Hence $T$ is WAP, and the assertions follow from Proposition
\ref{wap-zero} and Theorem \ref{wap}.
\end{proof}

\begin{prop} \label{wm}
Let $E$ be a Banach space with separable dual. Let $T$ be a contraction such that
$$
\frac1N \sum_{n=1}^N |\phi(T^nv)| \to 0\quad  \forall v\in E,\ \forall \phi \in E^*.
$$
Then $h^*_{top}(T)=0$.
\end{prop}
\begin{proof} By \cite[Corollary 4]{JL2}, the assumption yields that for every $v\in E$
there is a sequence $(n_j)$ of density 1 such that $T^{n_j}v \to 0$ weakly. By separability
of $E$ we can construct, as in the proof of \cite[Corollary 4]{JL2}, a sequence $(k_j)$
such that $T^{k_j}v \to 0$ weakly for every $v \in E$. This yields that for every
$\phi \in E^*$ we have $T^{*k_j}\phi \to 0$ weak*. Using the notations in the proof
of Proposition \ref{wap-zero}, we obtain that $S^{k_j}f(\phi) \to f(0)$ for every
$f \in C(B)$ and $\phi \in B$. This implies that $(B,T^*_{|B})$ is uniquely ergodic
with $\delta_0$ the unique $S$-invariant probability. Hence $h_{top}^*(T)=0$.
\end{proof}
{\bf Remarks.} 1. When $E$ is over $\mathbb C$, the assumption on $T$ is equivalent to
$T^*$ having no unimodular eigenvalues \cite{JL2}, and Theorem \ref{dor} applies.

2. By the proof of Theorem \ref{dor}, the assumption on $T$ in Proposition \ref{wm}
implies \eqref{strong-all} also in the real case.
\medskip

Recall that an operator $T$ on a Banach space $E$ is called {\it (weakly)  rigid} if
 for some subsequence $(n_k)$ we have $T^{n_k} v \to v$ (weakly) for every $v\in E$.


\begin{theorem} \label{rigid}
Let $T$ be a weakly rigid contraction on a separable Banach space $E$. Then $h^*_{top}(T)=0$.
\end{theorem}
\begin{proof}
Let $T$ be a weakly rigid contraction, with corresponding $(n_k)$.
By definition, $T^{*n_k}\phi \to \phi$ weak* for any $\phi \in B$.
We show that the restriction of $T^*$ to $B$ is a homeomorphism. If $T^*\phi_1=T^*\phi_2$,
then $T^{*n}\phi_1=T^{*n}\phi_2$ for every $n \ge 1$, so by rigidity $T^*$ is one-to-one.
Fix $\phi \in B$ and let a subsequence $T^{*(n_{k_j}-1)}\phi$ converge weak* to $\psi \in B$.
Then $T^*\psi=\phi$,  proving that $T^*$ is onto.

 In the notations of Proposition \ref{h*},
for $v \in E$ and $\phi \in B$ we have $f_v(T^{*n_k}\phi) \to f_v(\phi)$.
For $f \in C(B)$ define $Sf(\phi):=f(T^*\phi)$, and put
$$
A:=\{f\in C(B): S^{n_k}f (\phi)\to f(\phi)\quad \forall \phi \in B\}.
$$
It is obvious that $A$ is a linear manifold in $C(B)$, and is closed under multiplication.
 We show that $A$ is closed. Let $(f_j) \subset A$ converge to $f$. For $\varepsilon>0$ fix $j$
with $\|f_j-f\|< \varepsilon$. Then
$$|S^{n_k}f(\phi)-f(\phi)| \le
|S^{n_k}f(\phi)-S^{n_k}f_j(\phi)| + |S^{n_k}f_j(\phi)-f_j(\phi)| +|f_j(\phi)-f(\phi)|
$$
$$ \le |S^{n_k}f_j(\phi)-f_j(\phi)| +2\varepsilon.
$$
Hence $\limsup_k |S^{n_k}f(\phi)-f(\phi)| \le 2\varepsilon$, which yields $f \in A$. Constants
are in $A$, and the functions $f_v$, which are in $A$, separate points of $B$.
When $E$ is over $\mathbb C$ we note that $A$ is closed under conjugation.
By the Stone-Weierstrass theorem, $A=C(B)$.

Let $\nu=S^*\nu$ be an invariant probability. Then $\|S^{n_k}f-f\|_2 \to 0$ in $L_2(\nu)$,
for every $f \in C(B)$. Since $C(B)$ is dense in $L_2(\nu)$, we conclude that the system
$(B,T^*,\nu)$ is rigid, so $h_\nu(T^*)=0$ by \cite[Proposition F.9]{V}\footnote{Eli
Glasner proposed the following short proof for the ergodic case:  If $h_\nu(T^*)>0$,
by Sinai's factor theorem there is a Bernoulli factor of positive entropy; but the factor
is strongly mixing, and also rigid -- a contradiction. From this we obtain $h_\nu(T^*)=0$
for general invariant $\nu$ by the integral representation given in \cite{DGS}.}.
By the variational principle, $h_{top}(B,T^*)=0$, i.e. $h^*_{top}(T)=0$.
\end{proof}
\smallskip

{\bf Remarks.} 1. If $T$ is weakly rigid on $E \ne \{0\}$ complex, then its spectral radius is 1.
Indeed, if $T^{n_k} \to I$ in the weak operator topology, then by Banach-Steinhaus
$(\|T^{n_k}v \|)$ is bounded for every $v \in E$, and then $(\|T^{n_k}\|)$ is bounded,
say by $M$. By the spectral mapping theorem, $r(T)^{n_k}=r(T^{n_k}) \le M$, so
 $r(T) \le M^{1/n_k} \to_k 1$. If $r(T)<1$ then $\|T^n\| \to 0$,  contradicting
the weak rigidity. (See \cite[Proposition 2.18]{CMP} for $T$ rigid).
Note: a rigid $T$ need not be power-bounded \cite[p. 1352, Th\'eor\`eme]{MS}

2. If $T$ is power-bounded on $E$ separable, and $E$ is generated by the  eigenvectors of
unimodular eigenvalues, then $T$ is rigid \cite[Proposition IV.3.3]{E1} (or
\cite[Proposition 2.3]{E}).

3. If $T$ is a {\it weakly} rigid contraction on a (real or complex) Hilbert space
$\mathcal H$, then $T$ is rigid. Indeed, let $T^{n_k}v \to v$ weakly (we use only
$\langle T^{n_k}v,v\rangle \to \|v\|^2$); then
$\|v\|^2 \ge\|T^{n_k}v\|\cdot \|v\|\ge   |\langle  T^{n_k} v,v\rangle| \to  \|v\|^2$
yields $\|T^{n_k}v\|=\|v\|$ for every $k$, so $\|T^{n_k}v-v\|^2 \to 0$. By\cite{CMP}
$T$ is unitary. In general, weak rigidity does not imply rigidity --
see Proposition \ref{not-rigid}.

4. A rigid contraction need not be mean ergodic -- see Proposition \ref{not-ME}.


\begin{prop} \label{klr}
Let $(n_k)$ be an increasing sequence of  integers, satisfying
$$
\sup_k \Big\{ \underset{\overset{p|n_k}{p\  prime}} \sum \frac1p \Big\} <\infty.
$$
If $T$ is a power-bounded operator, weakly  rigid along $(n_k)$ (i.e. $T^{n_k}v \to v$ weakly
 for every $v \in E$), then  \eqref{sarnak-op-E} holds.
\end{prop}
\begin{proof} By changing to an equivalent norm we may assume $\|T\|\le 1$.

 We first prove for $E$ separable.
Let $B=\{\phi \in E^*: \|\phi\|\le 1\}$, which by separability is a compact metric space
in the weak* topology. As shown in the proof of Proposition \ref{rigid}, for every
probability $\nu$ invariant for $(B,T^*_{|B})$ and $f \in L_2(\nu)$ we have
$\|S^{n_k}f -f\|_2 \to 0$. By Kanigowski, Lema\'nczyk and Radziwi\l\l
\cite[Theorem 1.1]{KLR} (the invertibility assumption of \cite{KLR} was removed in \cite{W}), 
every $v \in E$ satisfies \eqref{sarnak*}, which yields \eqref{sarnak-op-E}.

When  $E$ is not separable, we fix $v \in E$ and look at the closed linear subspace spanned
by $(T^nv)_{n\ge 0}$, which is separable and $T$-invariant, and apply the previous result.
\end{proof}

In general, we have the following (weaker) result for rigid operators.
\begin{prop} \label{log}
Let $T$ be a power-bounded operator on $E$. Assume that for each $v \in E$ there exists
$(n_k)$ such that $\|T^{n_k}v-v\| \to 0$. Then
\begin{equation} \label{log-sarnak-op}
\frac1{\log N}\sum_{n=1}^N \frac{\mu(n) T^nv}n \to 0 \quad\text{\rm weakly in } \ E, \quad
\text{\rm for every }\ v \in E.
\end{equation}
\end{prop}
\begin{proof} As before, we may assume $\|T\| \le 1$. Fix $v \in E$, and let $E_v$ be the
closed subspace generated by $(T^nv)_{n\ge 0}$. Then $E_v$ is separable and $T$ invariant,
with the restriction of $T$ rigid on $E_v$, along the subsequence $(n_k)$ determined by $v$.

Let $B=\{\phi \in E_v^*: \|\phi\|\le 1\}$, which by separability is a compact metric space
in the weak* topology. As shown in the proof of Proposition \ref{rigid}, for every
probability $\nu$ invariant for $(B,T^*_{|B})$ and $f \in L_2(\nu)$ we have
$\|S^{n_k}f -f\|_2 \to 0$.  By Qiu, Wei and Xu \cite[Corollary 1.6]{QWX},
$$
 \frac1{\log N}\sum_{n=1}^N \frac{\mu(n)\phi(T^nv)}n =
 \frac1{\log N}\sum_{n=1}^N \frac{\mu(n)f_v(T^{*n}\phi)}n \to 0 \quad \forall \phi \in B,
 $$
which yields \eqref{log-sarnak-op}.
\end{proof}

\begin{prop} \label{kerr-li-1}
Let $T$ be an invertible isometry of a Banach space $E$.  If for every $v \in E$ the  orbit
$(T^nv)_{n\in \mathbb Z}$ does not contain a subset equivalent to the standard $\ell_1$ basis,
then $h^*_{top}(T)=0$.
\end{prop}
\begin{proof} Kerr and Li \cite{KL} defined a "CA-entropy" for invertible isometries of Banach
spaces, denoted by $hc(T)$. In \cite[Theorem 3.5]{KL} they prove that $hc(T)=0$ if and only
if $h^*_{top}(T)=0$. They also prove there that if $hc(T)>0$, there exists $v \in E$
such that $(T^nv)_{n\in \mathbb Z}$ contains a subset equivalent to the standard $\ell_1$ basis.
This yields the assertion.
\end{proof}

\begin{cor} \label{kerr-li}
Let $T$ be an invertible isometry of a Banach space $E$ which does not contain an isomorphic
copy of $\ell_1$. Then $h^*_{top}(T)=0$.
\end{cor}

{\bf Remarks.} 1. In \cite[Corollary 3.9]{KL} Kerr and Li conclude that $hc(T)=0$
when the Banach space has separable dual, since then it does not contain an
isomorphic copy of $\ell_1$. However, this latter condition does not imply
separability of $E^*$ when $E$ is separable; see \cite{LiS}.

2. A weakly rigid contraction $T$ is an invertible isometry by Proposition
\ref{rigid-isometry}, and then $h^*_{top}(T)=0$ by Theorem \ref{rigid}; when $E$
contains an isomorphic copy of $\ell_1$, Corollary \ref{kerr-li} does not apply.

3. If $T$ is invertible doubly power-bounded
(i.e $\sup_{n \in \mathbb Z} \|T^n\|<\infty$), then $T$ is an invertible isometry
in the equivalent norm $\||v\||:=\sup_{n\in\mathbb Z}\|T^n v\|$.
Such $T$ satisfies \eqref{strong} if and only if $T^{-1}$ does.

4.  Here is an alternative proof of  Corollary \ref{kerr-li} when $E$ is separable.
Then $X$, the unit ball of $E^*$, is compact metric in the weak* topology.
It was shown in the proof of Theorem \ref{rosenthal-rep} that $(X,\tau)$ is tame.
By Theorem 5.2 of Huang \cite{Hu}, for every $\tau$-invariant measure  $\nu$
 the measure preserving system $(X,\tau,\nu)$ has discrete spectrum.
Hence $h_\nu(\tau)=0$ (e.g. \cite[p. 255]{Pe}). By the variational principle,
  $h_{top}(\tau)=0$.
\medskip

We give the following properties of the class of contractions on $E$ with zero entropy.

\begin{prop}
Let $\mathcal{T}_0=\{T\in\mathcal L(E): \|T\|\le 1,\  h^*_{top}(T)=0\}$.
Then $\mathcal{T}_0$ is a convex set, stable under multiplication.
\end{prop}
\begin{proof} Let $T, S \in \mathcal{T}_0$ and $\alpha \in (0,1)$. 
Since  $\alpha T^* +(1-\alpha) S^*$ is a factor of $T^*\times S^*$,
$$
h^*_{top}(\alpha T +(1-\alpha) S) \leq h_{top}(T^*\times S^*)=h^*_{top}(T)+h^*_{top}(S)=0,
$$
the last equality by \cite[Theorem 3]{AKM}.

A similar argument yields that $h^*_{top}(TS)=0.$
\end{proof}

\subsection{Topological entropy for power-bounded operators using orbits.} \label{orbital}

Inspired by a blog of Tao \cite{T} and an early work of Veech \cite{V1}, we define
another type of entropy for power-bounded operators.

Let $\mathbf{a}:=(a_n)_{n\ge 0}$ be a bounded sequence, and let $S$ be the left shift
 on $\ell_\infty=\ell_\infty(\mathbb N)=\ell_1^*$, defined by $Sf(n)=f(n+1)$ for
$f \in \ell_\infty$. Let $K(\mathbf{a})$ be the weak*closure in $\ell_\infty$ of
$\{S^k\mathbf a\}$. Since $S$ is a dual operator (of the right shift in $\ell_1$),
it is continuous in the weak* topology of $\ell_\infty$, so $(K(\mathbf a),S)$ is a
metric topological dynamic system. We define the topological entropy of $\mathbf a$
by
$$
h_{top}(\mathbf a)=h_{top}\big((a_n)\big) := h_{top}(K(\mathbf a),S).
$$
\smallskip

Fix a power-bounded linear operator $T$ on a Banach space $E$. For $v\in E$ and
$\phi \in E^*$, we put
$$
a_n=a_n(v,\phi)=\langle T^nv,\phi\rangle, \quad n\ge 0.
$$
We denote $\mathbf a(v,\phi):=(a_n)_{n\in\mathbb N}$ and
$K(v,\phi):=K\big(\mathbf a(v,\phi)\big)$.
We  define the entropy of $v \in E$ as
$$
h_{top}(v)=\sup_{\phi \in E^*} h_{top}\big(\mathbf a(v,\phi)\big)=
\sup_{\phi \in E^*} h_{top}\big(K(v,\phi),S)\big).
$$


Finally, define the topological entropy of the linear operator $T$ as
$$
h_{top}(T)=\sup_{v \in E}h_{top}(v)=
\sup_{v\in E}\sup_{\phi \in E^*} h_{top}\big(\mathbf a\big(v,\phi\big)\big)=
\sup_{v\in E}\sup_{\phi \in E^*} h_{top}\big(K(v,\phi),S)\big).
$$
Since $a_n(v,\phi)=a_n(\alpha v,\alpha^{-1}\phi)$ for $\alpha>0$, we have
\begin{equation}\label{entropy2}
h_{top}(T)=
\sup_{v\in E}\sup_{\|\phi\|_{E^*}\le 1} h_{top}\big(\mathbf a\big(v,\phi\big)\big)=
\sup_{v\in E}\sup_{\|\phi\|_{E^*}\le 1} h_{top}\big(K(v,\phi),S)\big).
\end{equation}

\begin{prop} \label{h-T}
Let $T$ be a contraction on a Banach space $E$. If $h_{top}(T)=0$
and Sarnak's conjecture holds, then \eqref{sarnak-op-E} holds.
\end{prop}
\begin{proof}
Fix $v \in E$ and $\phi \in E^*$. By the definition, $h_{top}(T)=0$ implies that
the topological entropy of $(K(v,\phi),S)$ is zero. Let $e_0 \in \ell_1$ be
the unit vector $(1,0,\dots0,\dots)$, which is continuous on $K(v,\phi)$.
By the assumption that Sarnak's conjecture holds, we obtain, with
$\tilde Sf=f\circ S$ for $f$ continuous on $K(v,\phi)$,
$$
\frac1N\sum_{n=1}^N\mu(n)\langle \tilde S^n e_0,(b_k) \rangle \to 0  \quad
\text{ for any } (b_k) \in K(v,\phi),
$$
in particular for $(a_k):=(\langle T^kv,\phi\rangle)$.
Since $\langle \tilde S^ne_0,(a_k) \rangle = \langle e_0,S^n(a_k) \rangle =a_n$,
we have
$$
\frac1N\sum_{n=1}^N\mu(n) \langle T^{n}v,\phi \rangle =
\frac1N\sum_{n=1}^N \mu(n) a_{n}  \to 0,
$$
which yields \eqref{sarnak-op-E}.
\end{proof}

\begin{prop} \label{compare-h}
Let $T$ be a contraction on a Banach space $E$.Then $h_{top}(T) \le h^*_{top}(T)$.
\end{prop}
\begin{proof}
Let $B$ be the unit ball of $E^*$, and fix $v\in E$. Define $\pi_v: B \longrightarrow
\ell^\infty$ by $\pi_v(\phi)= (\phi(T^kv))_{k\in \mathbb N}$. Then, with $S$ the shift,
we have $\pi_v(T^*\phi)= \big(\phi(T^{k+1}v)\big)_{k\in\mathbb N}= S(\pi_v\phi)$.
Standard computations show that $\pi_v$ is continuous into the ball of radius $\|v\|$ in $\ell^\infty$ with the weak-* topology (see details in the proof of Proposition
\ref{separable}).  Then $\pi_v(B)$ is weak-* compact and $S$-invariant,  and coincides
with $K(v,\phi)$. Thus the system $\big(K(v,\phi),S\big)$ is a factor the system $(B,T^*)$,
so $h_{top}\big(a(v,\phi)\big) \le h^*_{top}(T)$. By \eqref{entropy2} we obtain
$h_{top}(T) \le h^*_{top}(T)$.
\end{proof}

\smallskip

{\bf Remarks.} 1. Proposition \ref{h*}, proved directly, is an immediate consequence of
Propositions \ref{compare-h} and \ref{h-T}.

2. For all the operators $T$ for which we previously proved that $h^*_{top}(T)=0$,
Proposition \ref{compare-h} yields that also $h_{top}(T)=0$.



%

\begin{prop} \label{minimal}
Let $(X,\tau)$ be a minimal topological dynamical system with $X$ metric, and let
$Tf=f\circ\tau$ be the induced contraction on $C(X)$.
If $h_{top}(T)=0$ then $h_{top}(\tau)=0$.
\end{prop}
\begin{proof} Assume $h_{top}(T)=0$. Then \eqref{entropy2}, with $v=f\in C(X)$ and
$\phi=\delta_x$, yields $h_{top}\Big(\big(f(\tau^n x)\big)_{n\ge 0}\Big)=0$; this for
every $f \in C(X)$ and $x\in X$.
	By Proposition 1.4 of Devianne \cite{De}, $h_{top}( \tau)=0$.
\end{proof}

\begin{cor} \label{all-zero}
Let $(X,\tau)$ be a minimal topological dynamical system with $X$ metric, 
and let $Tf=f\circ\tau$ be the induced operator on $C(X)$. 
Then the following are equivalent:

(i) $h_{top}(\tau)=0$.

(ii) $h_{top}(T)=0$.

(iii) $h^*_{top}(T)=0$.
\end{cor}
\begin{proof} Equivalence of (i) and (iii) is by Theorem \ref{metric}.
By Proposition \ref{compare-h},  (iii) implies (ii).

 By Proposition \ref{minimal} (ii) implies (i); this is the only implication which 
requires minimality.
\end{proof}

\medskip

\subsection{Previous definitions of entropy for operators}

 The notion of entropy of strictly positive-definite operators on  Hilbert space
was introduced by  Nakamura and Megaki \cite{NM}; it is given by $h(A) =-A\log(A)$,
 where $A$ is a strictly positive-definite operator on a Hilbert space.
But, in this case, for any contraction $A$ \eqref{sarnak-op-E} holds, even strongly,
 independently of the value of $h(A)$, by Theorem \ref{wap}.
\medskip

Downarowicz and Frej \cite[Section 4]{DF} defined a topological entropy for Markov operators
on $C(X)$ of a compact Hausdorff space $X$. They proved that when the Markov operator is 
given by $Tf=f\circ \tau$ with $\tau$ a continuous map of $X$ into itself, their topological 
entropy of $T$, which we denote $h^{DF}_{top}(T)$, equals $h_{top}(\tau)$. It follows that the 
condition $h^{DF}_{top}(T)=0$ is equivalent to conditions (i) and (iii) in Corollary \ref{all-zero}.
They also showed that if $P$ is a Markov operator on $C(X)$ such that $(P^n f)_{n\ge 0}$
converges uniformly for every $f \in C(X)$, then $h^{DF}_{top}(P)=0$; in that case, also
$h_{top}(P) \le h^*_{top}(P)=0$, by Propositions \ref{compare-h} and \ref{wap-zero}.
\medskip

Kerr and Li \cite{KL} defined a "CA-entropy" for {\it invertible}  isometries  (called there {\it
isometric automorphisms}) of Banach spaces. In \cite[Proposition 3.1]{KL} they showed that if
$\tau$ is a homeomorphism of a compact Hausdorff space and $Tf:=f\circ \tau$ is the induced
isometry on $C(X)$, then the CA-entropy of $T$, denoted by $hc(T)$, equals $h_{top}(\tau)$.
In \cite[Theorem 3.5]{KL}, Kerr and Li proved that (in our notation)  for every
invertible isometry $T$ on a Banach space, $hc(T)=0$ if and only if $h^*_{top}(T)=0$.
Since topological entropy can be define without invertibility,  our definition of
$h^*_{top}(T)$ applies to all contractions.
\medskip

Herrmann et al. \cite{HKL} applied a general definition of entropy for uniformly continuous
maps of non-compact metric spaces (due to Rufus Bowen \cite{Bow}) to define the topological 
entropy of a bounded linear operator $T$ on a Banach space. They proved that if $T$ is a 
contraction, then their entropy of $T$ is zero; therefore their operator entropy has no 
relevance to our problem, since there exist contractions (obtained from topological systems 
with positive topological entropy) which do not satisfy \eqref{sarnak-op-E}, e.g. \cite{DK},
\cite{AKL}, \cite{Kar}, \cite[Section 7]{AKLR1}.
\medskip

Chernyshev \cite{Che} discusses entropy for unitary operators on Hilbert space.
In view of Propositions \ref{H-contraction} and \ref{wap-zero}, \cite{Che} is
not relevant to the question of convergence of \eqref{sarnak-op-E}.

\bigskip

\section{Reductions to the shift on $\ell^\infty$}

\subsection{Reduction to the shift on $\ell^\infty(\mathbb Z)$}

Veech \cite{V2} studied Sarnak's conjecture using the left shift on the complex
$\ell^\infty:=\ell^\infty(\mathbb Z)$, defined by $Sf(n)=f(n+1)$ for
$f=\big(f(n)\big)_{n\in \mathbb Z} \in \ell^\infty$. In this subsection $B$ denotes
the unit ball of $\ell^\infty$ with the weak-* topology (induced by
$\ell^1(\mathbb Z)$), and $S_B$ is the restriction of $S$ to $B$.

%
%

Following Veech \cite{V2}, for $f=\big(f(k)\big)\in  \ell^\infty(\mathbb Z)$, let 
$\mathcal A_f$ be the weak-* closed *-subalgebra of $\ell^\infty$ generated by the orbit
$\{S^n f: n\in \mathbb Z\}$ and the constant sequences, and denote its maximal ideal
space (Gelfand space) by $X_f$. Then $X_f$ can be identified with the set of weak-* limits 
of subsequences of the orbit of $f$. Since $X_f$ is bounded, it is metrizable (and compact) 
in the weak-* topology. We then have a topological dynamical system $(X_f, S_{|X_f})$, and 
$C(X_f)$ is identified with $\mathcal A_f$ by Gelfand's theorem.

\begin{prop} \label{separable}
Let $T$ be an invertible isometry of a complex Banach space $E$ with separable dual $E^*$.
For fixed $v \in E$ and $\phi \in E^*$, define $f \in \ell^\infty$ by $f(k)=\phi(T^kv)$.
Then $X_f$ is separable in the $\ell^\infty$ norm.
\end{prop}
\begin{proof}
Let $K=\|\phi\|\cdot\|v\|$ and denote by $C$ the weak-* closure of 
$\{T^{*n}\phi: n\in\mathbb Z\}$ in $E^*$. Then $C$ is compact, and norm separable
 since $E^*$ is separable. Define $\pi: C\longrightarrow \ell^\infty$ by
$\pi(\psi)=\big(\psi(T^kv)\big)_{k\in\mathbb Z}$. Since $E$ is separable, the weak-* 
topology on bounded sets in $E^*$ is metrizable. Let $\psi_j\in C$ with $\psi_j\to\psi$
weak-*. Then for $g=(g_k)_{k\in \mathbb Z }\in\ell^1(\mathbb Z)$ we have
$\langle \pi(\psi_j), g\rangle = \sum_{k=-\infty}^\infty \psi_j(T^kv)g_k$.
Since weak-* convergence in $\ell^\infty$ is coordinate-wise convergence, we have
for the unit vectors  $e^\ell=(\delta_{k\ell})_{k\in\mathbb Z}$ of $\ell^1(\mathbb Z)$,
$$
\langle \pi(\psi_j),e^\ell\rangle= \psi_j(T^\ell v) \underset{j\to\infty}
 \longrightarrow \psi(T^\ell v) =\langle \pi(\psi),e^\ell \rangle.
$$
Hence $\pi$ is continuous from  $C$ into the ball of radius $K$ in $\ell^\infty$ with 
its weak-* topology. Then $\pi(C)$ is weak-* compact in $\ell^\infty$. Since $C$ is
$T^*$-invariant, the relationship $\pi(T^*\psi)= S\pi(\psi)$ shows that $\pi(C)$ is 
$S$-invariant, and contains the weak-* closure of 
$\big(S^n\pi(\phi)\big)_{n\in\mathbb Z}= (S^nf)_{n\in\mathbb Z}$.

Since $E^*$ is norm separable, so is $C$, and let $\{\phi_r\} \subset C$ be a countable 
set, norm dense in $C$. Fix $\psi \in C$ and let $\|\phi_{r_j}-\psi\| \to 0$. Then
$$
\|\pi(\phi_{r_j})-\pi(\psi)\|_\infty =  \sup_{\|g\|_1\le 1}
\Big|\sum_{k=-\infty}^\infty g_k\big(\phi_{r_j}(T^kv)-\psi(T^kv)\big) \Big|\le
$$
$$
\sup_{\|g\|_1\le 1}\sum_{k=-\infty}^\infty |g_k|\cdot |\phi_{r_j}(T^kv)-\psi(T^kv)|
\le \|\phi_{r_j}-\psi\|_{E^*}\|v\| \to 0,
$$
which shows that $\pi(C)$ is norm separable in $\ell^\infty$. Since $X_f \subset\pi(C)$, 
also $X_f$ is norm separable.
\end{proof}

Veech \cite[Theorem 9.1]{V2}  proved that when $X_f$ is norm separable, then
$h_{top}(X_f,S_{|X_f})=0$.
Huang, Wang and Ye \cite[Corollary 1.4]{HWY} proved that if $X_f$ is norm separable,
then
$$
\frac1N\sum_{n=1}^N \mu(n)f(n) \to 0.
$$
The proof given in \cite{HWY} was revised and completed by el Abdalaoui-Nerurkar 
in \cite{EN}.

\begin{cor}
Let $T$ be an invertible isometry of a complex Banach space $E$ with separable dual
 $E^*$. Then \eqref{sarnak-op-E} holds.
\end{cor}

{\bf Remark.} The assumption on $E$ in Theorem \ref{rosenthal-rep} is weaker.
\medskip
 
The class of $f \in \ell^{\infty}(\Z)$ for which $h_{top}(X_f,S_{|X_f})=0$ is denoted by 
$\textbf{SD}$ (for "shift deterministic").
By definition $\mathcal A_{Sf}=\mathcal A_f$ and $X_{Sf}=X_f$, so {\bf SD} is $S$-invariant.

\begin{theorem} \label{veech-conj}
 Sarnak's M\"{o}bius  randomness law conjecture for homeomorphisms is equivalent to
\begin{equation} \label{SD}
\frac1N\sum_{n=1}^N \mu(n)f(n) \to 0 \quad \text{\rm for every } f \in \textbf{SD}.
\end{equation}
\end{theorem}
\begin{proof} Assume Sarnak's conjecture \eqref{sarnak-op1} holds for all homeomorphisms of 
zero entropy.  Let $f \in \mathbf{SD}$. By definition  $h_{top}(X_f, S|_{X_f}) = 0$.
The function $\delta_0(g)=g(0)=\langle e_0,g\rangle$ is weak* continuous on $\ell^\infty$,
so the function $\phi : g \in X_f \mapsto \phi(g) = g(0)$ is in $C(X_f)$, and then 
$\phi(S^nf)=f(n)$. Applying \eqref{sarnak-op1} to $\phi$, with $\tau=S|_{X_f}$ and $x=f$, 
we obtain
$$
 \frac{1}{N}\sum_{n=1}^N \mu(n)f(n)= \frac{1}{N}\sum_{n=1}^N \mu(n)\phi(S^nf) \to 0.
$$
\medskip

{\bf Conversely,} assume \eqref{SD}.  Let $\tau$ be a homeomorphism of a compact space $X$,
with $h_{top}(X,\tau) = 0$, and fix $f \in C(X)$. For $x \in X$ define 
	$f_x \in \ell_\infty(\mathbb Z)$ by
\[
{f}_x(n) = f(\tau^n x), \quad n\in\mathbb Z.
\]
Let $Y:=\{f_x: x \in X\}$, which is bounded in $\ell^\infty$ by $\|f\|_\infty$.
The map $\pi x = f_x$ is continuous from $X$ onto $Y$ with the weak-* topology, so
$Y$ is compact. Since $S(\pi x)=f_{\tau x}=\pi(\tau x)$, $(Y,S_{|Y})$ is a factor of $(X,\tau)$.
Hence $h_{top}(Y,S_{|Y})=0$.

We now fix $x \in X$. Then the weak-* closure of 
$\{S^nf_x: n\in\mathbb Z\}= \{f_{\tau^n x}: n\in\mathbb Z\}$ is an $S$-invariant 
subset of $Y$, and so is $X_{{f}_x}$.  Hence, $h_{top}(X_{{f}_x},S_{|X_{{f}_x}})=0$, 
so ${f}_x \in \textbf{SD}$. Applying \eqref{SD}, we obtain
\[
 \frac{1}{N}\sum_{n=1}^{N}\mu(n) f(\tau^n x) = 
\frac{1}{N}\sum_{n=1}^{N}\mu(n){f}_x(n) \to 0.
\]
This completes the proof.
 \end{proof}

{\bf Remark.} The statement \eqref{SD} is a conjecture of Veech \cite[p. 1158]{V2}.
\medskip

For a subset $A\subset  \Z$ with $A\cap\mathbb N$ infinite, we identify $A$ with  
$f_A :=(1_A(k))_{k\in\Z}  \in \ell^\infty(\Z)$. Denote by $X_A$ the closure of the orbit,
 under the shift  in  $\{0,1\}^{\Z}$, of $(1_A(k))_{k\in\Z}$. 
 Then $(X_A,S_{|X_A})$ is isomorphic to  $(X_{f_A},S_{|X_{f_A}})$.
\begin{cor}
Let $A$ be a subset of $\mathbb Z$ with $A\cap\mathbb N$ infinite, such that the subshift 
$(X_A,S_{|X_A})$ has topological entropy zero. If Sarnak's M\"{o}bius  randomness law 
holds, then 
\begin{equation} \label{subset}
\frac{1}{N}  \underset{\overset{1\le n \le N}{n\in A}}\sum\mu(n) \tend{n}{+\infty}0.
\end{equation}
\end{cor}
\begin{proof} By the assumption, $f_A \in \textbf{SD}$, so \eqref{SD} applies to $f_A$ by 
Theorem \ref{veech-conj}.
\end{proof}

{\bf Remark.} When $A\cap\mathbb N$ has zero  density, \eqref{subset} holds independently 
of Sarnak's conjecture. Since $\frac1N\sum_{n=1}^N \mu(n) \to 0$, \eqref{subset} holds 
also whenever $A\cap\mathbb N$ has density 1. 
 By Dirichlet's Theorem \cite[Chapter 7]{Apo}, the corollary applies when $d > 1$
 and $A = \{r + nd : n \in \mathbb Z\}$ for some $r\in \{0,\dots , d - 1\}$:
$A\cap\mathbb N$ has density $1/d$, while $X_A$ is finite so has entropy 0.
\medskip

\subsection{Reduction to the shift on $\ell_\infty(\mathbb N)$}

The above approach of Veech \cite{V2} applies to invertible isometries of Banach spaces, as 
seen in Proposition \ref{separable} and its corollary. However, in general operators are not 
invertible, so in order to deal with contractions we should use $\ell_\infty(\mathbb N)$ 
instead of $\ell^\infty(\mathbb Z)$. Now $S$ denotes the left shift on the complex 
$\ell_\infty:=\ell_\infty(\mathbb N)$, defined by $Sf(n)=f(n+1)$ for 
$f=\big(f(n)\big)_{n\ge 0} \in \ell_\infty=\ell_\infty(\mathbb N)$.
 In this subsection $B$ denotes the unit ball of $\ell_\infty$ with the weak-*
 topology, and $S_B$ is the restriction of $S$ to $B$.

Adapting the approach of Veech, for $f=\big(f(k)\big)_{k \ge 0}\in \ell_\infty(\mathbb N)$, 
let $\mathcal A_f$ be the weak-* closed *-subalgebra of $\ell_\infty$ generated by the orbit
$\{S^n f: n\ge 0\}$ and the constant sequences, and denote its maximal ideal
space (Gelfand space) by $X_f$. Then $X_f$ can be identified with the set of weak-* limits 
of subsequences of the orbit of $f$. Since $X_f$ is bounded, it is metrizable (and compact) 
in the weak-* topology. We then have a topological dynamical system $(X_f, S_{|X_f})$, and 
$C(X_f)$ is identified with $\mathcal A_f$ by Gelfand's theorem.

\begin{theorem} \label{veech-conjecture}
The following statements are equivalent:

(i) For every $f\in \ell_\infty(\mathbb N)$ with $h_{top}(X_f,S_{|X_f})=0$,
$\frac1N\sum_{n=1}^N \mu(n)f(n) \to 0$.

(ii) For every compact metric topological dynamical system $(X,\tau)$, $\tau$ not 
necessarily invertible, with $h_{top}(X,\tau)=0$, \eqref{sarnak-op} holds.

(iii) Veech's conjecture \eqref{SD} holds.
\end{theorem}
\begin{proof}
Assume (ii). For $f \in \ell_\infty(\mathbb N)$, $\delta_0(f):=\langle e_0,f\rangle=f(0)$ 
defines a weak* continuous  function on $\ell_\infty(\mathbb N)$, with 
$\delta_0(S^nf) =f(n)$. Since $\delta_0$ is weak* continuous,  $\delta_0 \in C(X_f)^*$. 
 When $h_{top}(X_f,S_{|X_f})= 0$, we apply (ii) to $(X_f,S_{|X_f})$ and obtain that
$$
\frac1N \sum_{n=1}^N \mu(n)f(n) =
	\frac1N \sum_{n=1}^N \mu(n)\delta_0(S^nf)=
\langle \e_0, \frac1N \sum_{n=1}^N \mu(n)S^nf\rangle \to 0.
$$
Hence (i) holds.
\smallskip

Assume (i). Let $(X,\tau)$ be a metric topological system  with $h_{top}(X,\tau) = 0$. 
	We proceed as in the proof of Theorem \ref{veech-conj}.
	Fix $f \in C(X)$. For $x \in X$ define $f_x \in \ell_\infty(\mathbb N)$ by
\[
{f}_x(n) = f(\tau^n x), \quad n \in\mathbb N.
\]
Let $Y:=\{f_x: x \in X\}$, which is bounded in $\ell^\infty$ by $\|f\|_\infty$.
The map $\pi x = f_x$ is continuous from $X$ onto $Y$ with the weak-* topology, so
$Y$ is compact. Since $S(\pi x)=f_{\tau x}=\pi(\tau x)$, $(Y,S_{|Y})$ is a factor of $(X,\tau)$.
Hence $h_{top}(Y,S_{|Y})=0$.

We now fix $x \in X$. Then the weak-* closure of 
$\{S^nf_x: n\in\mathbb N\}= \{f_{\tau^n x}: n\in\mathbb N\}$ is an $S$-invariant subset of $Y$, 
and so is $X_{{f}_x}$.  Hence, $h_{top}(X_{{f}_x},S_{|X_{{f}_x}})=0$. 
Applying (i), we obtain
\[
 \frac{1}{N}\sum_{n=1}^{N}\mu(n) f(\tau^n x) = 
\frac{1}{N}\sum_{n=1}^{N}\mu(n){f}_x(n) \to 0.
\]
Thus the system satisfies \eqref{sarnak-op1}, which is equivalent to \eqref{sarnak-op}.
This completes the proof of equivalence of (i) and (ii).
\smallskip

By Theorem \ref{veech-conj}, (ii) implies (iii). We prove that (iii) implies (ii).
 Let $(X,\tau)$ be a metric topological system  with $h_{top}(X,\tau) = 0$. If $\tau$ is 
invertible, \eqref{sarnak-op} and \eqref{sarnak-op1} hold by Theorem 
\ref{veech-conj}, so we assume that $\tau$ is not invertible. Let $(\hat X,\hat \tau)$
be the extension  topological system with a homeomorphism $\hat\tau$ given by Proposition 
\ref{extension}. By Proposition \ref{both0}, also $h_{top}(\hat X,\hat\tau)=0$, and by
Theorem \ref{veech-conj} the system $(\hat X,\hat \tau)$ satisfies \eqref{sarnak-op1}.
Let $\pi$ be the factor map of $(\hat X,\hat \tau)$ onto $(X,\tau)$. Given $f \in C(X)$,
define $\hat f=f\circ\pi$. Since $\pi$ is onto, for $x \in X$ there is $\hat x \in\hat X$ 
with $\pi\hat x=x$. Since  $(\hat X,\hat \tau)$ satisfies \eqref{sarnak-op1}, we obtain
$$
\frac1N\sum_{n=1}^N \mu(n) f(\tau^n x) =
\frac1N\sum_{n=1}^N \mu(n) f(\tau^n \pi\hat x)=
\frac1N\sum_{n=1}^N \mu(n) \hat f(\hat\tau^n \hat x) \to 0;
$$
thus $(X,\tau)$ satisfies \eqref{sarnak-op1} and therefore also \eqref{sarnak-op}.
This completes the equivalence of (ii) and (iii).
 \end{proof}

{\bf Remark.} Condition (iii) of Theorem \ref{veech-conjecture} is Veech's  conjecture
\cite{V2}; condition (i) is its adaptation to $\ell_\infty(\mathbb N)$.
Condition (ii) is Sarnak's conjecture.


\begin{cor} If condition  (i) of Theorem \ref{veech-conjecture} holds, then for every contraction
 $T$ on a separable Banach space $E$ with $h^*_{top}(T)=0$, \eqref{sarnak-op-E} holds.
\end{cor}
\begin{proof}
Combine Theorem \ref{veech-conjecture} with Proposition \ref{h*}.
\end{proof}

\bigskip

\section{Pointwise convergence for contractions in $L^p$ spaces}

Let $(X,\Sigma,m)$ be a probability space, $1 \le p<\infty$, and $T$ an operator on $L^p(m)$.
In this section we study the almost everywhere convergence of the "weighted" averages
$\frac1N\sum_{n=1}^N \mu(n)T^nf (x)$. When $1<p< \infty$ and $T$ is power-bounded, we
have norm convergence by Theorem \ref{wap}.

Let $\tau: X \mapsto X$ be measure preserving, and for $1\le p< \infty$ put $Tf=f\circ\tau$.
Sarnak \cite[page 6]{sarnak} observed that Bourgain's results in \cite{Bo} yield that
for $f \in L^2(m)$,
\begin{equation} \label{a-e}
\frac1N\sum_{n=1}^N \mu(n)T^nf (x) \to 0  \quad \text{a.e. }
\end{equation}
El Abdalaoui et al. \cite[Proposition 3.1]{AKLR1}, observed that this yields that
\eqref{a-e} holds for every $f \in L^1(m)$, using the spectral theorem and Davenport's
estimate \eqref{davenport}.  Independently, Cuny and Weber \cite[Corollary 2.14]{CW}
deduced \eqref{a-e} for $f \in L^1(m)$ from the Banach principle, since by Birkhoff's
ergodic theorem, every $f \in L^1(m)$ satisfies
$$
\sup\big|\frac1N\sum_{n=1}^N \mu(n)T^nf (x)\big| \le
\sup\frac1N\sum_{n=1}^N T^n|f| (x)  <\infty \quad\text{a.e.}
$$

{\bf Remark.} Note that the  convergence  \eqref{a-e} for every $f \in L^1$
 makes no assumption on the entropy of $\tau$.

\smallskip

{\bf Definition.}
 A contraction $T$ on $L^1(X,\Sigma,m)$ of a probability space is called  a
 {\it Dunford-Schwartz} contraction if $\|Tf\|_\infty \le \|f\|_\infty$ for every
 $f \in L^\infty(m)$ (i.e.  $T$ is also a contraction of $L^\infty(m)$).
 A Dunford-Schwartz contraction is then a contraction of every $L^p(m)$ space,
 $1<p<\infty$ (by the Riesz-Thorin theorem).

\begin{theorem} \label{ds}
Let $(X,\Sigma,m)$ be a probability space and $T$ a Dunford-Schwartz
contraction on $L^1(m)$. Then
\begin{equation} \label{ds-a-e}
\frac1N\sum_{n=1}^N \mu(n)T^nf (x) \to 0  \quad \text{a.e.}, \quad \forall f \in L^1(m)
\end{equation}
\end{theorem}
\begin{proof} Baxter and Olsen \cite[Theorem 2.19]{BO} proved that if $(a_n)$ is a bounded
sequence such that $\frac1N \sum_{n=1}^N a_nT^nf(x)$ converges a.e for every $T$ induced
on $L^1(m)$ by an invertible measure preserving transformation and every $f\in L^1(m)$,
then the same convergence holds for every Dunford-Schwartz operator. In view of the above
results of \cite{CW} and \cite{AKLR1}, we have that $\frac1N \sum_{n=1}^N \mu(n)T^nf(x)$
converges a.e for every Dunford-Schwartz $T$ and $f \in L^1$. Since for  $f \in L^2$
we have $\big\|\frac1N \sum_{n=1}^N \mu(n)T^nf(x)\big\|_1\le
\big\|\frac1N \sum_{n=1}^N \mu(n)T^nf(x)\big\|_2 \to 0$ by Theorem \ref{wap},
$\big\|\frac1N \sum_{n=1}^N \mu(n)T^nf(x)\big\|_1\to 0$ for any $f \in L^1$, and the a.e.
limit is zero by Fatou's lemma.
\end{proof}

\begin{cor} \label{markov}
Let $(X,\Sigma)$ be a measurable space and let $P(x,A)$ be a transition probability
on $X\times\Sigma$. Assume that $m$ is an invariant probability for the Markov operator
$P$, defined for $f$ bounded measurable by $Pf(x)=\int f(y)P(x,dy)$;
i.e. $m(A)= \int P(x,A) dm(x)$ for $A \in \Sigma$. Then $P$ extends to a contraction of
	$L^1(m)$, and
\begin{equation} \label{markov-a-e}
\frac1N\sum_{n=1}^N \mu(n)P^nf (x) \to 0  \quad \text{a.e.},\ \forall f \in L^1(m).
\end{equation}
\end{cor}
\begin{proof} By invariance, $\int Pf\,dm=\int f\,dm$ for $f \in L^\infty(m)$.
Hence $P$ is a positive Dunford-Schwartz operator on $L^1(m)$.
\end{proof}

{\bf Remark.} Under the assumptions of Corollary \ref{markov}, we can obtain
\eqref{markov-a-e} for $f \in L^\infty(m)$ from Bourgain's result  by Kakutani's method
for the pointwise ergodic theorem \cite{Ka} (see also \cite{D}), and deduce the
convergence for $L^1$ functions by the Banach principle as in \cite{CW}, using
Hopf's pointwise ergodic theorem \cite[Section 10]{H} instead of Birkhoff's.

\begin{theorem}
Fix $p \in (1,\infty)$.  Let $(X,\Sigma,m)$ be a probability space and $T$ a positive
contraction of $L^p(m)$. Then
$$
\frac1N\sum_{n=1}^N \mu(n)T^nf (x) \to 0  \quad \text{a.e.}, \quad \forall f \in L^p(m).
$$
\end{theorem}
\begin{proof} Almost everywhere convergence of $\frac1N\sum_{n=1}^N \mu(n)T^nf (x)$
follows from Theorem \ref{ds} and \cite[Theorem 2.4]{CLO}. The limit is zero a.e.
by Fatou's lemma, since $L^p$-norm convergence to zero holds  by Theorem \ref{wap}.
\end{proof}
\medskip

We now look at contractions on spaces of Bochner integrable vector valued  functions.
Let $(X,\Sigma,m)$ be a probability space and $E$ a Banach space. For $1\le p <\infty$
we denote by $L^p(X;E)$ the space of (equivalence classes of) strongly measurable
$E$-valued functions $f:X\mapsto E$ such that
$$
\|f\|_p:= \Big(\int_X \|f(x)\|_E^p\ dm\Big)^{1/p} < \infty.
$$
The space of strongly measurable $f: X\mapsto E$ with
$\|f\|_\infty:=\text{ess }\sup_{x \in X} \|f(x)\|_E < \infty$ is denoted by $L^\infty(X;E)$. For more details see \cite[Section V.5]{Y}, \cite[p. 167]{K}.
Since we assume $m$ to be a probability,
$L^\infty(X;E) \subset L^p(X;E) \subset L^1(X;E)$ for $1<p< \infty$.

\begin{prop} \label{vector-me}
Let $(X,\Sigma,m)$ be a probability space, $E$ a reflexive Banach space, and $1<p<\infty$.
If $T$ is a power-bounded operator on $L^p(X;E)$, then for every $f \in L^p(X;E)$ we have
$$
\Big\| \frac1N\sum_{n=1}^N \mu(n)T^nf\Big\|_p =
\Big(\int_X \Big\| \frac1N\sum_{n=1}^N \mu(n)T^nf(x)\Big\|_E^p\ dm \Big) ^{1/p}
\longrightarrow 0.
$$
\end{prop}
\begin{proof} Phillips \cite[Theorem 5.7]{Ph} proved that for $1<p< \infty$, reflexivity of
$E$ implies that of $L^p(X;E)$. The claimed convergence now follows from Theorem  \ref{wap}.
\end{proof}

\begin{prop}
Let $T_0$ be power-bounded on a Banach space $E$, and for $f \in L^1(X;E)$
define $Tf$ by $(Tf)(x):=T_0(f(x))$. Then $T$ is power-bounded on $L^p(X;E)$,
$1\le p\le \infty$.

 When $E$ is reflexive,  every $f \in L^1(X;E)$ satisfies
\begin{equation}
\Big\| \frac1N \sum_{n=1}^N \mu(n) T^n f(x) \Big\|_E \to 0 \quad \text{a.e.}
\end{equation}
\end{prop}
\begin{proof}
By induction we obtain that $(T^nf)(x)=T_0^n(f(x))$.  Hence  \newline
$\|T^nf(x)\|_E=\|T_0^n(f(x))\|_E \le C\|f(x)\|_E$, $x \in X$.
\smallskip

For $f \in L^p(X;E)$, $1\le p< \infty$, we then have
$$
\|T^n f\|_p^p=\int_X \|T_0^n(f(x))\|_E^p\ dm  \le C^p \int_X \|f(x)\|_E^p\ dm =
C^p \| f\|_p^p,
$$
For $f \in L^\infty(X;E)$ we  have  $\|T^n f\|_\infty \le C\|f\|_\infty$.
\smallskip

When $E$ is reflexive, Theorem \ref{wap}, applied to $T_0$, yields
$$
\Big\| \frac1N \sum_{n=1}^N \mu(n) T^n f(x) \Big\|_E =
\Big\| \frac1N \sum_{n=1}^N \mu(n) T_0^n (f(x)) \Big\|_E \to 0, \quad x\in X.
$$
\end{proof}

\begin{theorem} \label{chacon-0}
Let $S$ be a positive Dunford-Schwartz contraction on $L^1$ of a probability space
$(X,\Sigma,m)$ and $E$ a Banach space. Then there exists a contraction $T$ of $L^1(X;E)$
such that 
$$
T(\sum_{j=1}^J v_j{\bf 1}_{A_j})=\sum_{j=1}^J v_jS{\bf 1}_{A_j}, \qquad 
	v_1,\dots, v_J \in E, \quad A_1,\dots,A_J \in \Sigma,
$$
	and every $f \in L^1(X;E)$ satisfies
\begin{equation} \label{pre-chacon0}
\Big\| \frac1N \sum_{n=1}^N \mu(n) T^n f(x) \Big\|_E \to 0 \quad \text{a.e.}
\end{equation}
\end{theorem}
\begin{proof} Let $f$ be finitely valued, $f =\sum_{j=1}^J v_j{\bf 1}_{A_j}$, with
$v_j \in E$ and disjoint $A_j$. Then $f \in L^1(X;E)$, and we define 
$Tf(x):=\sum_{j=1}^J v_jS{\bf 1}_{A_j}(x)$, which is obviously strongly measurable.  Then
$$
\|Tf\|_1=\int \Big\|\sum_{j=1}^J v_jS{\bf 1}_{A_j}(x)\Big\|_E\ dm \le
$$
$$
\sum_{j=1}^J \|v_j\|_E \int |S{\bf 1}_{A_j}(x)|\ dm \le
\sum_{j=1}^J \|v_j\|_E\  m({A_j})=\|f\|_1.
$$
By the above, when $f_k$ are finitely valued in $L^1(X;E)$, with 
$f_k \to f \in L^1(X;E)$ in norm, $(Tf_k)$ is a Cauchy sequence in norm; the limit defines 
$Tf$, and does not depend on the sequence $(f_k)$. This yields 
$\|Tf\|_1 =\lim \|Tf_k\|_1 \le \lim \|f_k\|_1=  \|f\|_1$.
	\smallskip

{\it Claim 1: Let $g \in L^\infty(X,m)$ and $v \in E$. Then $v\,g(x) \in L^1(X;E)$
and $T(vg)(x)=v\,Sg(x)$ a.e.}\newline
For $g$  with finitely many values this is by the definition of $T$.  Let $g_k \to g$
uniformly on $X$ with $g_k$ finitely valued. Then $Sg_k \to Sg$ uniformly, and then
$v\,g_k(x) \to v\,g(x)$ a.e. and in $L^1(X;E)$-norm, so $T(v\,g)=\lim T(v\,g_k)=
\lim v\,Sg_k=v\,Sg$, which proves the claim.
\smallskip

{\it Claim 2: For $f \in L^1(X;E)$ we have $\|T^nf(x)\|_E \le S^n\|f\|_E(x)$ a.e.}\newline
It is enough to prove for $n=1$. For $f =\sum_j v_j {\bf 1}_{A_j}$ with $A_j$ disjoint
we have $\|f(x)\|_E= \|v_j\|$ when $x \in A_j$. Hence
$\|f(x)\|_E =\sum_{j=1}^J \|v_j\|{\bf 1}_{A_j}(x)$. By Claim 1 and positivity of $S$,
$$
\|Tf(x)\|_E  = \Big\|\sum_j v_j S{\bf 1}_{A_j}(x)\Big\|_E  \le
\sum_j \|v_j\|_E\  S{\bf 1}_{A_j}(x) =
$$
$$
S\big(\sum_j \|v_j\|_E\  {\bf 1}_{A_j}\big)(x) = S\|f\|_E(x).
$$
This proves the claim for finitely valued $f$. For general $f \in L^1(X;E)$, let $f_k$ 
be finitely valued with $f_k \to f$ a.e. and in $L^1(X;E)$-norm. Then 
$\|f_k\|_E \to \|f\|_E$ a.e. and in $L^1(m)$-norm, so 
$S\|f_k\|_E \to S\|f\|_E$ in $L^1(m)$; hence for some subsequence, still denoted $f_k$, 
we have $S\|f_k\|_E(x) \to S\|f\|_E(x)$ a.e. Then
$\|Tf(x)\|_E =\lim_k \|Tf_k(x)\|_E \le \lim_k S\|f_k\|_E(x) = S\|f\|_E(x)$ a.e.,
and Claim 2 is proved. 
\medskip

Claim 1 yields $T^n(v\,g)=vS^ng$ a.e. Let $f=\sum_{j=1}^J v_j{\bf 1}_{A_j}$; then
$$
\Big\| \frac1N \sum_{n=1}^N \mu(n) T^n f(x) \Big\|_E=
\Big\| \sum_{j=1}^J v_j\frac1N \sum_{n=1}^N \mu(n) \ S^n{\bf 1}_{A_j}(x) \Big\|_E \le
$$
\begin{equation} \label{simple-f0}
\sum_{j=1}^J \|v_j\|_E\ \Big|\frac1N \sum_{n=1}^N \mu(n) \ S^n{\bf 1}_{A_j}(x)\Big|
\underset{N\to\infty}\longrightarrow 0 \quad \text{a.e.},
\end{equation}
 by Theorem \ref{ds}. For the general case,
 let $f_k$ be finitely valued with $f_k(x) \to f(x)$ a.e. and in $L^1(X;E)$-norm. Then
$$
\limsup_N \Big\| \frac1N \sum_{n=1}^N \mu(n) T^n f(x) \Big\|_E \le
$$
$$
\limsup_N \Big\| \frac1N \sum_{n=1}^N \mu(n) T^n(f-f_k)(x)\Big\|_E +
\limsup_N \Big\| \frac1N \sum_{n=1}^N \mu(n) T^n f_k(x) \Big\|_E.
$$
As $N\to\infty$, the last term converges to 0  a.e. by \eqref{simple-f0}. Claim 2 yields
$$
\Big\| \frac1N \sum_{n=1}^N \mu(n) T^n(f-f_k)(x))\Big\|_E \le
\frac1N \sum_{n=1}^N S^n\|f-f_k\|_E( x) .
$$
By the Dunford-Schwartz pointwise  ergodic theorem \cite[Theorem VIII.6.6]{DS}, the last
 expression converges a.e. as $N \to \infty$ to  an integrable function $g_k$; since $S$ is
positive, $\int g_k dm\ \le \int \|f-f_k\|_E\ dm < \varepsilon$ for large $k$.
Together, we obtain $\big\| \frac1N \sum_{n=1}^N \mu(n) T^n f(x) \big\|_E \to 0$ a.e.
\end{proof}

{\bf Remarks.} 1. The operator $T$ in the  Theorem  is also a contraction of $L^\infty(X;E)$.

2. Note that there is no assumption of reflexivity in Theorem \ref{chacon-0}.
\smallskip

The next theorem is a corollary of Theorem \ref{chacon-0}. However, we present its proof
since it is somewhat simpler. It is a vector-valued version of
the $L^1$-extension of Bourgain-Sarnak's \eqref{a-e} \cite{AKLR1},\cite{CW}.

\begin{theorem} \label{chacon0}
Let $\tau$ be a measure preserving transformation of a probability space $(X,\Sigma,m)$
and $E$ a Banach space. For $f \in L^1(X;E)$ define $Tf=f\circ \tau$. Then $T$ is
a contraction of $L^1(X;E)$, and every $f \in L^1(X;E)$ satisfies
\begin{equation} \label{pre-chacon}
\Big\| \frac1N \sum_{n=1}^N \mu(n) T^n f(x) \Big\|_E=
\Big\| \frac1N \sum_{n=1}^N \mu(n)  f(\tau^n x) \Big\|_E \to 0 \quad \text{a.e.}
\end{equation}
\end{theorem}
\begin{proof} It is easily shown that $f(\tau x)$ is strongly measurable. Let $f$ be
finitely valued, $f =\sum_{j=1}^J v_j{\bf 1}_{A_j}$, with $v_j \in E$ and $A_j$ disjoint. Then
	$\|f(x)\|_E=\sum_{j=1}^J\|v_j\|{\bf 1}_{A_j}(x)$, and
$$
\|Tf\|_1=\int \Big\|\sum_{j=1}^J v_j{\bf 1}_{A_j}(\tau  x)\Big\|_E\ dm \le
$$
$$
\sum_{j=1}^J \|v_j\|_E \int {\bf 1}_{A_j}(\tau  x) dm =
\sum_{j=1}^J \|v_j\|_E\  m({A_j})=\|f\|_1.
$$
By the above, when $f_k$ are finitely valued with $f_k \to f$ in norm, $(Tf_k)$ is a
 Cauchy sequence in norm, and the limit is $Tf$; this yields $\|Tf\|_1 \le \|f\|_1$.

Since $T^nf(x)=f(\tau^nx)$, for $f=\sum_{j=1}^J v_j{\bf 1}_{A_j}$ we obtain
\begin{equation} \label{simple-f}
\Big\| \frac1N \sum_{n=1}^N \mu(n) T^n f(x) \Big\|_E=
\Big\| \frac1N \sum_{n=1}^N \mu(n) \sum_{j=1}^J v_j{\bf 1}_{A_j}(\tau^n x) \Big\|_E \le
$$
$$
\sum_{j=1}^J \|v_j\|_E\ \Big| \frac1N \sum_{n=1}^N \mu(n) {\bf 1}_{A_j}(\tau^n x)\Big|
	\underset{N\to\infty}\longrightarrow  0 \quad \text{a.e.},
\end{equation}
 by \cite{CW} (or by Theorem \ref{ds}). For the general case,
 let $f_k$ be finitely valued with $f_k(x) \to f(x)$ a.e. and in $L^1(X;E)$-norm. Then
$$
\limsup_N \Big\| \frac1N \sum_{n=1}^N \mu(n) T^n f(x) \Big\|_E \le
$$
$$
	\limsup_N \Big\| \frac1N \sum_{n=1}^N \mu(n) T^n(f-f_k)(x))\Big\|_E +
\limsup_N \Big\| \frac1N \sum_{n=1}^N \mu(n) T^n f_k(x) \Big\|_E.
$$
As $N\to\infty$, the last term converges to 0  a.e. by \eqref{simple-f}. Now
$$
	\Big\| \frac1N \sum_{n=1}^N \mu(n) T^n(f-f_k)(x))\Big\|_E \le
\frac1N \sum_{n=1}^N \|f(\tau^n x)-f_k(\tau^n x)\|_E\ ;
$$
By Birkhoff's ergodic theorem, the last expression converges a.e. as
$N \to \infty$ to  an integrable function $g_k$ with
$\int g_k dm=\int \|f-f_k\|_E\ dm < \varepsilon$ for large $k$. Together, we obtain
$ \big\| \frac1N \sum_{n=1}^N \mu(n) T^n f(x) \big\|_E \to 0$ a.e.
\end{proof}

{\bf Remark.} The operator $T$ in Theorem \ref{chacon0} is also a contraction of
$L^\infty(X;E)$.
\medskip

Chacon \cite{Ch} extended the Dunford-Schwartz pointwise ergodic theorem
to contractions $T$ of $L^1(X;E)$ of $E$ reflexive which contract also the
$L^\infty(X;E)$-norm\footnote{As in \cite{DS}, Chacon assumes $m$ to be $\sigma$-finite,
not necessarily a probability.}. He proved that for every $f \in L^p(X;E)$, $1\le p< \infty$,
the averages $\frac1N\sum_{n=1}^N T^nf(x)$ converge strongly in $E$ for a.e. $x \in X$.
For $1<p<\infty$, Proposition \ref{vector-me} yields, by Fatou's lemma, that
$$
\liminf_{N\to \infty} \Big\|\frac1N \sum_{n=1}^N \mu(n) T^hf(x)\Big\|_E = 0 \quad \text{a.e.}
$$
The question is whether, under Chacon's assumptions, every $f \in L^1(X;E)$ satisfies
\eqref{pre-chacon0}.

\begin{theorem} \label{L(H)}
Let $(X,\Sigma,m)$ be a probability space.
Let $\mathcal H$ be a Hilbert space and $T$ a contraction of $L^1(X;\mathcal H)$ which
contracts the $L^\infty(X;\mathcal H)$-norm. Then every $f \in L^\infty(X;\mathcal H)$ satisfies
\begin{equation} \label{chacon-H}
\Big\| \frac1N \sum_{n=1}^N \mu(n) T^n f(x) \Big\|_\mathcal H \to 0 \quad
 \text{for a.e.}\ x\in X.
\end{equation}
\end{theorem}
\begin{proof} We first note that $L^2(X;\mathcal H)$ is a Hilbert space, with the inner product
$$
\langle f,g\rangle := \int \langle f(x),g(x)\rangle_\mathcal H\ dm(x), \quad
f,g \in L^2(X;\mathcal H).
$$
By the Riesz-Thorin theorem (as in \cite{Ch}) $T$ is a contraction of $L^2(X;\mathcal H)$.
Applying  Proposition \ref{H-contraction} we obtain
$\| \frac1N \sum_{n=1}^N \mu(n) T^n\|_2 \le C_\alpha/(\log N)^\alpha$ for any $\alpha>0$.
We take $\alpha=1$, so
$$
\Big\| \frac1N \sum_{n=1}^N \mu(n) T^n f\Big\|_2 \le \frac{C_1\|f\|_2}{(\log N)}, \quad
f\in L^2(X;\mathcal H).
$$
Fix $f\in L^2(X;\mathcal H)$, and let $\rho>1$. Taking $N_m=[\rho^m]+1$ we obtain
$$
\Big\|\frac1{N_m} \sum_{n=1}^{N_m} \mu(n) T^n f\Big\|_2^2 \le \frac{C\|f\|_2^2}{(m\log \rho)^2}\ ,
$$
which yields the convergence of the series
$$
\sum_{m\ge 1}
 \int \Big\| \frac1{N_m} \sum_{n=1}^{N_m} \mu(n) T^n f(x) \Big\|_\mathcal H^2\ dm=
\sum_{m\ge 1} \Big\| \frac1{N_m} \sum_{n=1}^{N_m} \mu(n) T^n f \Big\|_2^2 < \infty.
$$
By Beppo Levi's theorem, we obtain
$$
 \Big\| \frac1{N_m} \sum_{n=1}^{N_m} \mu(n) T^n f(x) \Big\|_\mathcal H \to 0
\quad \text{a.e.}
$$
We now assume $f \in L^\infty(X;\mathcal H)$. Let $[\rho^m]<N \le [\rho^{m+1}]+1$.
Then $N\ge N_m> \rho^m$, and using $\|T^nf(x)\|_\mathcal H \le \|f\|_\infty$ a.e., we get
$$
\Big\| \frac1{N} \sum_{n=1}^{N} \mu(n) T^n f(x) \Big\|_\mathcal H =
$$
$$
\Big\| \frac1{N} \sum_{n=1}^{N_m} \mu(n) T^n f(x) +
 \frac1{N} \sum_{n=N_m+1}^{N} \mu(n) T^n f(x) \Big\|_\mathcal H \le
$$
$$
\Big\| \frac1{N_m} \sum_{n=1}^{N_m} \mu(n) T^n f(x) \Big\|_\mathcal H +
\frac{(\rho^{m+1}+1 -\rho^m)\|f\|_\infty}{\rho^m}
 \underset{m\to\infty}\longrightarrow (\rho-1)\|f\|_\infty\quad \text{a.e.}
$$
Hence $\limsup_N\big\| \frac1{N} \sum_{n=1}^{N} \mu(n) T^n f(x) \big\|_\mathcal H
\le (\rho-1)\|f\|_\infty$ a.e. Letting $\rho \to 1^+$ we obtain \eqref{chacon-H}.
\end{proof}

{\bf Definition.} A (real or complex) Banach lattice\footnote{We refer to
\cite[Sections 1 and 3]{La} for the definitions and properties.} $\mathcal L$ is called
a {\it Hilbert lattice} if it is a Hilbert space, i.e. if the norm satisfies the
parallelogram equality.

\begin{prop} \label{hilbert-lattice}
Let $\mathcal H$ be a separable real Hilbert space and $(e_n)_{n\ge1}$ an orthonormal
basis. Then $\mathcal H$ is a Hilbert lattice in the partial order defined by
$$
\sum_{n=1}^\infty a_ne_n  \succeq \sum_{n=1}^\infty b_n e_n \quad \text{\rm if and only if}
\quad a_n \ge b_n\  \forall n\ge 1.
$$
With this order $\mathcal H$ is isometric and order isomorphic to the real $\ell_2$.
\end{prop}

{\bf Definition.} A linear operator $T$ on a Banach lattice $\mathcal L$ with partial order
$\preceq$ is called {\it positive} if $v \succeq 0$ implies $Tv \succeq 0$.

 \begin{prop} \label{Lp-lat}
Let $\mathcal L$ be a real Banach lattice and $1\le p < \infty$. Then $L^p(X;\mathcal L)$,
with the partial order $f\underset{p}\succeq g$ if and only if $f(x) \succeq g(x)$
for every $x \in X$, is a Banach lattice.
 \end{prop}
 \begin{proof} The order on $\mathcal L$ yields that for $f,g \in L^p(X;\mathcal L)$
 we have $f\underset{p}\vee g$, which is given by $(f\underset{p}\vee g)(x)=f(x)\vee g(x)$.
Thus $L^p(X;\mathcal L)$ is a vector lattice. It is easy to see that $|f|(x)=|f(x)|$.
 In order to show that $L^p(X;\mathcal L)$ is a Banach lattice, we need to show
that $|f| \underset{p}\succeq |g|$ implies $\|f\|_p \ge \|g\|_p$.
Let $|f| \underset{p}\succeq |g|$  in $L^p(X;\mathcal L)$; then $|f(x)| \succeq |g(x)|$
for every $x$, which yields $\|f(x)\|_\mathcal L \ge\|g(x)\|_\mathcal L$ for every $x \in X$,
 so $\|f\|_p \ge \|g\|_p$.
 \end{proof}

 A linear operator $T$ of $L^p(X;\mathcal L)$ is  positive if $(Tf)(x) \succeq 0$
for every $x \in X$ when  $f(x)\succeq 0$ for every $x \in X$.

\begin{theorem} \label{L1(H-lat)}
Let $(X,\Sigma,m)$ be a probability space, let $\mathcal L$ be a real Hilbert lattice,
and $T$ a positive contraction of $L^1(X;\mathcal L)$ which contracts the
$L^\infty(X;\mathcal L)$-norm. Then every $f \in L^1(X;\mathcal L)$ satisfies
\begin{equation} \label{chacon-lat}
\Big\| \frac1N \sum_{n=1}^N \mu(n) T^n f(x) \Big\|_\mathcal L \to 0 \quad \text{a.e.}
\end{equation}
\end{theorem}
\begin{proof} We first note that since $T$ is also a contraction of $L^2(X;\mathcal L)$,
it is mean ergodic in $L^2(X;\mathcal L)$, and density of $L^2(X;\mathcal L)$ in
$L^1(X;\mathcal L)$ yields that $T$ is mean ergodic in $L^1(X;\mathcal L)$.

By Proposition \ref{Lp-lat}, all $L^p(X;\mathcal L)$, $1 \le p < \infty$, are Banach lattices, 
and $T$ is a positive contraction of each of them.

Let $g \in L^1(X;\mathcal L)$. Then by positivity $\pm T^ng \underset{1}\preceq T^n|g|$, so
	$\pm\mu(n)T^n g \underset{1}\preceq T^n|g|$. This yields
$\pm \Big( \frac1N\sum_{n=1}^N\mu(n)T^ng\Big) \underset{1}\preceq
 \frac1N\sum_{n=1}^N T^n|g|$, so
\begin{equation} \label{dominate}
\Big| \frac1N\sum_{n=1}^N\mu(n)T^ng\Big| \underset{1}\preceq
 \frac1N\sum_{n=1}^N T^n|g|.
\end{equation}
By Chacon's theorem \cite{Ch}, $ \frac1N\sum_{n=1}^N T^n|g|(x)$ converges in $\mathcal L$-norm
for almost every $x \in X$. Hence, \eqref{dominate} and Fatou's Lemma yield
$$
\int_X \limsup_{N\to\infty}\Big\| \frac1N\sum_{n=1}^N\mu(n)T^ng(x) \Big\|_\mathcal L\,dm \le
$$
$$
\int_X \lim_{N\to\infty}\Big\| \frac1N\sum_{n=1}^N T^n|g|(x) \Big\|_{\mathcal L}\,dm \le
\liminf_{N\to\infty}\int_X\Big\| \frac1N\sum_{n=1}^N T^n|g|(x) \Big\|_{\mathcal L}\,dm =
$$
\begin{equation} \label{limsup}
\liminf_{N\to\infty} \Big\| \frac1N\sum_{n=1}^N T^n|g| \Big\|_1\,dm  \le
\|g\|_1.
\end{equation}

Let $f \in L^1(X;\mathcal L)$. By mean ergodicity $f=f_1+f_2$, with $Tf_1=f_1$ and
$f_2 \in\overline{(I-T)L^1(X;\mathcal L)}$. By density of $L^\infty(X;\mathcal L)$,
for $k\ge 1$ there is $h_k \in L^\infty(X;\mathcal L)$ such that $g_k:=f_2-(I-T)h_k$
satisfies $\|g_k\|_1 <\frac1k$. Since $\mathcal L$ is a Hilbert space and $\|T\|_\infty \le 1$,
Theorem \ref{L(H)} yields
$\big\|\frac1N\sum_{n=1}^N \mu(n)T^n(I-T)h_k(x)\big\|_\mathcal L \to 0$ as $N\to\infty$.
Hence
$$
\limsup_{N\to \infty} \Big\|\frac1N\sum_{n=1}^N \mu(n)T^n f(x)\Big\|_\mathcal L \le
$$
$$
\limsup_{N\to \infty} \Big\|\frac1N\sum_{n=1}^N \mu(n)T^n f_1(x)\Big\|_\mathcal L +
\limsup_{N\to \infty} \Big\|\frac1N\sum_{n=1}^N \mu(n)T^n g_k(x)\Big\|_\mathcal L=
$$
$$
\limsup_{N\to \infty} \Big\|\frac1N\sum_{n=1}^N \mu(n)T^n g_k(x)\Big\|_\mathcal L \ ,
$$
since $ \big\|\frac1N\sum_{n=1}^N \mu(n)T^n f_1(x)\big\|_\mathcal L 	=
\|f_1(x)\|_\mathcal L \big|\frac1N\sum_{n=1}^N \mu(n)\big| \to 0$ a.e. as $N\to \infty$.
Using \eqref{limsup} we then have
$$
\int_X \limsup_{N\to \infty} \Big\|\frac1N\sum_{n=1}^N \mu(n)T^n f(x)\Big\|_\mathcal L\, dm\le
$$
$$
\int_X\limsup_{N\to \infty} \Big\|\frac1N\sum_{n=1}^N \mu(n)T^n g_k(x)\Big\|_\mathcal L\, dm
 \le \|g_k\|_1 < \frac1k\ .
$$
Letting $k\to\infty$ we obtain
 $\int_X \limsup_{N} \big\|\frac1N\sum_{n=1}^N \mu(n)T^n f(x)\big\|_\mathcal L\, dm=0$,
which proves the theorem.
\end{proof}

{\bf Remarks.} 1. If the norm in a Banach lattice $\mathcal L$ satisfies
${\|u+v\|^2}=\|u\|^2+\|v\|^2$ whenever $u\wedge v=0$, then $\mathcal L$ is linearly isometric
and order isomorphic to an $L^2$-space (Bohnenblust-Nakano theorem \cite[p. 135]{La}).
Hence $\mathcal L$ is a Hilbert lattice.

2. If $\mathcal L$ is a real Banach lattice such that $\|u+v\|^2 +\|u-v\|^2=2(\|u\|^2+\|v\|^2)$
for every $u,v \succeq 0$, then $\mathcal L$ is a Hilbert lattice \cite[Lemma 5]{El}.
\medskip

\section{Problems}

In this section we present problems that arose in the present research.
\medskip

{\bf Problem 1.} {\it Let $E$ be a Banach space. If every power-bounded $T$
satisfies \eqref{sarnak-op-E}, is $E$ reflexive?}

A candidate for a negative answer would be the order 1 quasi-reflexive space of James,
which is "closest" to reflexive. Another candidate is a space with "few" operators,
where every $T$ is of the form  "scalar plus compact"; there are several spaces
of this type.

All the above spaces have a Schauder  basis, so by \cite{FLW} have power-bounded
operators which are not mean ergodic.

Note that in the James space, every invertible $T$ with 
$\sup_{n\in \mathbb Z} \|T^n\|< \infty$ satisfies \eqref{sarnak-op-E}, by Theorem 
\ref{rosenthal-rep}.
\medskip

{\bf Problem 2.} {\it Let $T$ be a rigid contraction on a separable Banach space $E$.
Does $T$ satisfy \eqref{sarnak-op-E}?}

By Theorem  \ref{rigid}, $h^*_{top}(T)=0$, so if Sarnak's conjecture is true, then the
answer is positive by Proposition \ref{h*}. By Proposition \ref{log}, any $T$ rigid satisfies
\eqref{log-sarnak-op}. By Corollary \ref{rigid-sarnak}, the answer is positive when $E$
does not contain an isomorphic copy of $\ell_1$. A special case of the question is when
$T$ is a rigid contraction on $\ell_1$.
\medskip

{\bf Problem 3.} {\it Is Theorem \ref{dor2} true when $T^*$ has countably many unimodular
eigenvalues?}

Let $R_\sigma(T)$ be the residual spectrum of $T$. by the Hahn-Banach theorem it is the
set of eigenvalues of $T^*$. Thus, the assumption on $T$ means that
$R_\sigma(T)\cap \mathbb T $ is countable. A special case is when $E^*$ is separable;
then $E$ does not contain an isomorphic copy of $\ell_1$, and $T^*$ has at most countably
many unimodular eigenvalues, by a theorem of Jamison \cite{Jami}. Spectral conditions on a
power-bounded $T$ which yield strong convergence of $T^n$ (hence almost periodicity of $T$)
were studied by Arendt and Batty \cite{AB} and by Katznelson and Tzafriri \cite{KT}.

In view of Theorem \ref{rosenthal-rep}, this question is a special case of the question
whether every power-bounded operator on a space $E$ which does not contain an isomorphic
copy of $\ell_1$ satisfies \eqref{sarnak-op-E}.
\medskip

{\bf Problem 4.} {\it Let $\tau$ be a continuous map of a compact metric space $X$ and
$Tf=f\circ\tau$ for $f \in C(X)$. Is $h_{top}^*(T)=h_{top}(\tau)$?}

By Proposition \ref{tau-on-X} we have $h_{top}^*(T) \ge h_{top}(\tau)$. 
By Theorem \ref{metric}, the answer to the problem is positive in the special case that 
$h_{top}(\tau)=0$, thus extending \cite[Theorem A]{GW}.
\smallskip

The Downarowicz-Frej topological entropy $h^{DF}_{top}(T)$ equals $h_{top}(\tau)$, 
so a related question is whether every Markov operator $P$ on $C(X)$ satisfies 
$h^{DF}_{top}(P) \le h^*_{top}(P)$.
\medskip

{\bf Problem 5.} {\it Let $T$ be a quasi-compact contraction. Is $h^*_{top}(T)=0$?}

This question is motivated by Proposition \ref{quasi-compact}. The answer is positive
when $T^n$ is compact for some $n$: $T$ is then WAP by Corollary \ref{compact},
so $h^*_{top}(T)=0$  by  Proposition \ref{wap-zero}.
\medskip

{\bf Problem 6.} {\it Let $T$ be a contraction. If $h_{top}(T)=0$, does
$h^*_{top}(T)=0$?}

It was shown in Proposition \ref{compare-h} that $h_{top}(T) \le h^*_{top}(T)$,
so the question is whether  zero entropy is equivalent in both definitions. A related
(harder) question (not relevant to the Sarnak conjecture) is whether
$h_{top}(T) = h^*_{top}(T)$.

\medskip

{\bf Problem 7.} {\it Let $\tau$ be a continuous map of a compact metric space $X$ and
$Tf=f\circ\tau$ for $f \in C(X)$. If $h_{top}(T)=0$, is $h_{top}(\tau)=0$?}

The answer is positive when $\tau$ is minimal, by Proposition \ref{minimal}. If the
answer is positive wtihout minimality, then the  minimality assumption can be removed
from Corollary \ref{all-zero}.
\medskip

{\bf Problem 8.}  {\it Under the assumptions of Theorem \ref{L(H)}, does \eqref{chacon-H}
hold for every $f \in L^1(X;\mathcal H)$?}

A positive answer will  include Theorem \ref{L1(H-lat)} as Corollary. A related question,
motivated by Chacon's theorem \cite{Ch}, is whether $\mathcal H$ in Theorem \ref{chacon-H}
can be replaced by any reflexive Banach space $E$.
\bigskip

\section{Appendix A: Weakly rigid operators}

In this appendix we discuss weak rigidity of power-bounded operators on a Banach space.
We recall the definition: a bounded linear operator $T$ on a Banach space $E$ is called
{\it (weakly) rigid} if there exists an increasing sequence $(n_j)$ such that
$T^{n_j}v \to v$  (weakly) for every $v \in E$.

Rigid operators were introduced by Furstenberg and Weiss \cite{FW}
for the Koopman operators induced on $L^2$ by probability preserving transformations.
Weak rigidity was introduced for topological dynamical systems by Glasner and Maon \cite{GM}
(under the name "rigidity", since it yields rigidity in $L^2$ of any invariant probability).
Rigid operators were studied by Costakis et al. \cite{CMP}.

\begin{prop} \label{not-rigid}
There exists a weakly rigid  contraction on $C(X)$ of a compact metric space
which is not rigid.
\end{prop}
\begin{proof}
K\"orner \cite{Ko} constructed a compact metric space $X$ with metric $d$ and a continuous
map $ \tau:\, X \longrightarrow X$ such that the the system $(X,\tau)$ is minimal,
\begin{equation} \label{not-uniform}
\sup_{x \in X} d(\tau^n x,x) \ge 1 \,\text{ for every }\, n\ge 1,
\end{equation}
and there exists an increasing sequence $(n_j)$ such that
\begin{equation} \label{weak}
\tau^{n_j}x \to x\, \text{ for every } x \in X.
\end{equation}
The operator $Tf:=f\circ\tau$ on $C(X)$ is then weakly rigid by \eqref{weak}.

Since $d (x,y)$ is continuous in $(x,y)$, for  $n\ge 1$ $\ d(\tau^nx,x)$
is continuous, so by \eqref{not-uniform} there exists $z_n \in X$
such that $d(\tau^n z_n,z_n)=\sup_{x \in X} d(\tau^n x,x) \ge 1$.
\smallskip

By compactness, for any increasing sequence $(k_\ell)$ there is a subsequence
$(n_i) \subset(k_\ell)$ such that $z_{n_i}$ converges to some $y \in X$.
Put $g(x):=d(x,y)$, so $T^ng(x)=d(\tau^nx,y)$. Then
$$
1 \le d(\tau^{n_i}z_{n_i},z_{n_i}) \le d(\tau^{n_i}z_{n_i},y) +d(z_{n_i},y)=
$$
$$
T^{n_i}g(z_{n_i}) -g(z_{n_i}) + 2d(z_{n_i},y).
$$
Hence $\|T^{n_i}g -g\|_\infty \ge T^{n_i}g(z_{n_i}) -g(z_{n_i}) \ge 1- 2d(z_{n_i},y)\to 1$.
	This shows that $T$ is not rigid on $C(X)$.
\end{proof}
	
If $T$ is a weakly rigid {\it  contraction}, then it is an isometry, since for $v \in E$
and $\phi \in E^*$ we have
$$
|\phi(v)| = \lim_k|\phi(T^{n_k}v)| \le \lim_k \|T^{n_k}v\| \cdot \|\phi\| \le \|Tv\|\cdot\|\phi\|,
$$
which yields $\|v\|= \sup_{\|\phi\| \le 1} |\phi(v)| =\|Tv\|$.

\begin{prop} \label{rigid-isometry}
Let $T$ be a weakly rigid power-bounded operator on $E$. Then $T$ is invertible, and
$T^{-1}$ is also power-bounded.
\end{prop}
\begin{proof}
The operator $T$ is a contraction in the equivalent norm $\||v\||=\sup_{n\ge 0}\|T^nv\|$,
so it is an isometry. Hence $T$ is one-to-one.

Fix $v \in E$. Since $T^{n_k}v \to v$ weakly, the sequence $(T^{n_k}v)_{k\ge 1}$ is weakly
sequentially compact. By The Krein-Shmulian theorem \cite[p. 434]{DS}, the closed convex hull
of $(T^{n_k}v)$ is weakly compact, so contains $v$. But then $v$ is a strong limit of finite
convex combinations $\sum c^{(j)}_kT^{n_k}v$, which are all in the image $T(E)$. Hence
$T$ is onto, and therefore invertible. In the equivalent norm $T$ is an isometry, so also
$T^{-1}$ is an isometry, so power-bounded in the original norm.
\end{proof}

Glasner and Maon \cite{GM} called a continuous map $\tau$ of a compact metric space
$X$ into itself {\it uniformly rigid} if for some increasing $(n_k)$ we have
$\tau^{n_k}x \to x$ for every $x \in X$ uniformly. It was proved in \cite[Lemma 2.2]{JKLS}
that $\tau$ is uniformly rigid on $X$ if and only if the operator $Tf:= f\circ\tau$ on
$C(X)$ is rigid.

\begin{prop} \label{not-ME}
There exists a rigid contraction on $C(X)$ of a compact metric space
which is not mean ergodic.
\end{prop}
\begin{proof} Glasner and Weiss (private communication) constructed $\tau$ which
is minimal, uniformly rigid, and not uniquely ergodic. Then $Tf=f\circ\tau$ is rigid
on $C(X)$ \cite{JKLS}. Since $\tau$ is minimal and not uniquely ergodic, $T$ is
	not mean ergodic \cite[p. 180]{K} (or \cite[Theorem 2.2]{S}).
\end{proof}

The following is included for completeness; we leave its proof to the reader.
\begin{prop}
Let $T$ be a contraction on a separable complex Banach space $E$ and $(n_i)_{i>0}$
an increasing sequence of natural numbers. If, for some $1 \ne \lambda \in \mathbb T$,
we have $\|T^{n_i}v - \lambda v\| \to 0$ for every $v \in E$, then $T$ is rigid.
\end{prop}

\bigskip

\section{Appendix B (by Christophe Cuny): Power-bounded operators on $\mathcal H$}\label{CC}

Power-bounded operators on a Hilbert space are weakly almost periodic, so Theorem
\ref{wap} applies, and an operator norm convergence holds for contractions by
Proposition \ref{H-contraction}.  The following improves both these results.

\begin{theorem}\label{haase}
Let $T$ be power-bounded on a Hilbert space $\mathcal H$. Then
\begin{equation}
\Big\| \frac1N \sum_{n=1}^{N} \mu(n) T^n\Big\|\le \frac{K_\alpha}{(\log N)^\alpha}.
\qquad \forall \alpha> 0.
\end{equation}
\end{theorem}
\begin{proof} We first prove the theorem when $\mathcal H$ is over $\mathbb C$.

 For $N \ge 1$, denote $\mu_N=\frac1N(\mu(1),\dots,\mu(N),0,0,\dots) \in\ell_1(\N)$,
  define the polynomial
$\hat\mu_N(z):= \frac1N \sum_{k=1}^N \mu(k)z^k,\ |z|\le 1$, and put
$\hat\mu_N(T):=  \frac1N \sum_{k=1}^N \mu(k)T^k$. The support of $\mu_N$ is a
subset of $[1,N]$, so, with $M:= \sup_n\|T^n\|$, Haase's \cite[Theorem 4.6]{Ha},
with $p=2$ and $X=\mathcal H$, yields
$$
\|\hat\mu_N(T)\| \le c_2(1+\log N)M^2 \sup_{|z|\le1}|\hat\mu_N(z)|,
$$
using the fact (explicit in \cite{Ha} for $C_0$-semigroups) that for Hilbert spaces
the multiplier norm is the sup norm.

For $\alpha >0$ we use Davenport's \eqref{davenport} with $\epsilon=1+\alpha$ and
obtain
$$
\|\hat\mu_N(T)\| \le C_{1+\alpha}c_2(1+\log N)M^2 \frac1{(\log N]^{1+\alpha}},
$$
which yields $\|\hat\mu_N(T)\| \le K_\alpha/(\log N)^\alpha$.
\smallskip

When  $\mathcal H$ is over $\mathbb R$, we use complexification, as in the proof
of Proposition \ref{H-contraction}.
\end{proof}
\bigskip

\section{Appendix C: Extension of non-invertible topological systems   }

Many results in topological dynamics are proved for topological systems with a homeomorphism.  
In this appendix we extend  systems  with non-invertible transformations  (which appear, for
example, in the definition of $h^*_{top}(T)$ for contractions) to become factors of a system 
with homeomorphism. The purpose of the extension is to enable extension of results from 
systems with homeomorphism to the general case.
\smallskip

The main tool will be the "natural extension" of non-invertible probability preserving 
transformations, defined by Rokhlin \cite{R1},\cite[pp. 12-13]{R}, presented in the following.

\begin{theorem} \label{rokhlin}
Let $\tau$ be a non-invertible measure preserving transformation (endomorphism) of a
Lebesgue probability space $(X,\Sigma,m)$. Then there exists a "minimal" {\rm invertible}
measure preserving transformation $\hat\tau$ on a probability space 
$(\hat X,\hat\Sigma,\hat m)$ (automorphism), such that the system $(X,\Sigma,m,\tau)$
 is  factor of the system $(\hat X,\hat\Sigma,\hat m,\hat\tau)$.

Moreover, $\hat\tau$ is ergodic if and only if $\tau$ is.
\end{theorem}

Rokhlin \cite[p. 28]{R} proved that the Kolmogorov-Sinai entropies satisfy 
$h_m(\tau)=h_{\hat m}(\hat\tau)$.

\begin{prop} \label{extension}
Let $X$ be a compact metric space and $\tau$ a non-invertible continuous map of $X$ into 
itself. Then there exists a homeomorphism $\hat\tau$ of a compact metric space $\hat X$ 
such that for every $\tau$-invariant probability $\nu$ on $X$ there exists a $
\hat\tau$-invariant probability $\hat\nu$ on $\tilde X$ so that 
$(\hat X, \hat\tau, \hat\nu)$ is the natural extension of $(X,\tau,\nu)$.
\end{prop}

The Proposition follows from the construction of the natural extension: the system
$(\hat X,\hat\tau)$ is defined only by $(X,\tau)$, and $\hat X$ is compact (metric)
when $X$ is.
\smallskip

\begin{prop} \label{both0}
Let $(X,\tau)$ and $(\hat X,\hat\tau)$ be as in Proposition \ref{extension}.
Then  $h_{top}(\hat\tau) =h_{top}(\tau)$.
\end{prop} 
\begin{proof} Under the assumptions of Proposition \ref{extension}, the variational 
principle \eqref{goodman} and the equality $h_\nu(\tau)=h_{\hat\nu}(\hat\tau)$ yield
 $h_{top}(\tau) \le h_{top}(\hat\tau)$.
\smallskip

Let $\pi$ be the factor map of $(\hat X,\hat\tau)$ onto $(X,\tau)$. Let $\eta$
be a $\hat\tau$-invariant probability, and define $\nu(A)=\eta(\pi^{-1}A)$ for Borel 
sets of $X$. Then $(\hat X,\hat\tau,\eta)$ is a natural extension of $(X,\tau,\nu)$,
 so $\eta=\hat\nu$, and $h_\nu(\tau)=h_{\hat\nu}(\hat\tau)$. Then
 $h_{top}(\tau)\ge h_\nu(\tau)=h_{\hat\nu}(\hat\tau)=h_\eta(\hat\tau)$. Since $\eta$
was arbitrary, \eqref{goodman} yields  
$h_{top}(\hat\tau) =\sup_{\{\eta =\eta\hat\tau^{-1}\} } h_\eta(\hat\tau)\le h_{top}(\tau)$.
\end{proof}

\bigskip

\section{Appendix D: Modulation by some $\pm 1$ sequences}

In most of the proofs in Section \ref{modulated}, and also  in Section \ref{CC}, 
the M\"obius function was used only through Davenport's inequality \eqref{davenport}. 
In this appendix we look at modulation by ceratin $\pm 1$ sequences.

\subsection{Modulation by the Liouville function}

The Liouville function is defined for $n\ge 1$ by $\lambda(n):=(-1)^{\Omega(n)}$, where
$\Omega(n)$ is the number of prime factors of $n$, counted with multipilicity. Similarly
to the M\"obius function, it is of great importance in number theory.

Bateman and Chowla \cite[Lemma 1]{BC} deduced from Davenport's \eqref{davenport} 
the estimate
\begin{equation} \label{liouville}
\max_{\theta \in [0,1)}\left|\displaystyle\sum_{k \leq t}\lambda(k)e^{2\pi ik\theta}\right| 
\le C_\varepsilon \frac{t}{\log(1+t)^{\varepsilon}}, \qquad\text{\rm for any}\ \varepsilon >0.
\end{equation}
We can therefore replace the M\"obius function  in Section  \ref{modulated} and in 
Theorem \ref{haase} by the Liouville function.

\subsection{Modulation by the Rudin-Shapiro sequence}

An inequality similar to \eqref{davenport} was proved  for the Rudin-Shapiro sequence 
constructed in \cite{Ru}: a real sequence $(r(n)_{n\ge 1}$ with $|r(n)|=1$, which satisfies
\begin{equation} \label{RS}
\max_{\theta \in [0,1)}\left|\displaystyle\sum_{k=1}^N r(k)e^{2\pi ik\theta}\right|
\le 5N^{1/2}.
\end{equation}
We can therefore replace the M\"obius sequence in most theorems of Section
 \ref{modulated} and in Theorem \ref{haase} by the Rudin-Shapiro sequence.
 We state two examples of particular interest.

\begin{theorem}
Let $T$ be weakly almost periodic on a Banach space $E$. Then
\begin{equation}
\Big\| \frac1N\sum_{n=1}^{N} r(n) T^nv\Big\| \to 0 \qquad \text{\rm for every }\ v\in E.
\end{equation}
\end{theorem}

\begin{theorem}
Let $T$ be power-bounded on a Hilbert space. Then
\begin{equation}
\Big\| \frac1N \sum_{n=1}^{N} r(n) T^n\Big\|\le \frac{C\log N}{N^{1/2}}.
\end{equation}
\end{theorem}
\medskip

\subsection{Modulation by random signs.}  \label{random}

Let $(X_n(t))_{n\ge 1}$ be the Rademacher sequence defined on $[0,1]$ (with Lebegue 
measure), which is an iid sequence of random variables taking the values +1 or -1 with 
probability $\frac12$. From Salem and Zygmund \cite[Theorem 4.3.1]{SZ} (see also 
\cite{Hal}), we obtain
\begin{equation} \label{salem}
\max_{\theta\in [0,1]}\left| \sum_{k=1}^N X_k(t) \cos(k\theta) \right| \le C_t\sqrt{N\log N}
\quad \text{for a.e. }t\in[0,1].
\end{equation}
As observed in \cite{SZ}, the result holds also with phase shifts, so we conclude that
\begin{equation} \label{zygmund}
\max_{\theta\in [0,1]}\left| \sum_{k=1}^N X_k(t) \e^{2\pi ik\theta} \right| \le 
	C_t\sqrt{N\log N} \quad \text{for a.e. }t\in[0,1].
\end{equation}
See also \cite{SZ1}. Then, for a.e. $t$, we can replace the M\"obius sequence by 
$(X_n(t))_n$ in most of the results of Section 2. We obtain, for example,

\begin{theorem}
Let $T$ be weakly almost periodic on a Banach space $E$. Then, for a.e. $t \in [0,1]$,
\begin{equation}
\Big\| \frac1N\sum_{n=1}^{N} X_n(t) T^nv\Big\| \to 0 \qquad \text{\rm for every }\ v\in E.
\end{equation}
\end{theorem}

\begin{theorem}
Let $T$ be power-bounded on a Hilbert space. Then, for a.e. $t \in [0,1]$,
\begin{equation}
\Big\| \frac1N \sum_{n=1}^{N} X_n(t) T^n\Big\|\le \frac{C_t(\log N)^{3/2}}{N^{1/2}}.
\end{equation}
\end{theorem}

It is known \cite[p. 191]{K} that if $T$ is a contraction on $L^2$ of a probability space
which is not positive, then the pointwise ergodic theorem for $T$ may fail. Randomizing
the signs yields the following, by \cite[Corollary 4.3]{CL05} and Kronecker's lemma.

\begin{theorem} \label{L2-contractions}
For a.e. $t \in [0,1]$, the sequence $(X_n(t))_{n\ge 1}$ has the following property:

For every contraction (not necessarily positive) on $L^2(X,\Sigma,m)$ of a probability space 
and any $f \in L^2(m)$,
$$
\frac1N\sum_{n=1}^N X_n(t)T^nf(x)= 
o\Big(\frac{(\log N)^2(\log\log N)^{1/2+\epsilon}}{N^{1/2}}\Big) \quad m\text{\rm -a.e.}
$$
\end{theorem}

\begin{cor} \label{random-DS}
For a.e. $t \in [0,1]$, the sequence $(X_n(t))_{n\ge 1}$ has the following property:

For every Dunford-Schwartz contraction on $L^1(m)$ of a probability space $(X,\Sigma,m)$
and any $f \in L^1(m)$,
$$
\frac1N\sum_{n=1}^N X_n(t)T^nf(x)  \to 0 \quad m\text{\rm -a.e.}
$$
\end{cor}
\begin{proof} Let $\mathbf{T}$ be the linear modulus of $T$ \cite[p. 159]{K}, which is
a positive Dunford-Schwartz contraction. Then for any $f \in L^1(m)$ the Dunford-Schwatz
ergodic theorem yields
$$
\sup_{N\ge 1}\Big|\frac1N \sum_{n=1}^N X_n(t) T^nf\Big| \le
\sup_{N\ge 1}\frac1N \sum_{n=1}^N \mathbf{T}^n|f| < \infty\quad m\text{\rm-a.e.}
$$
This and Theorem \ref{L2-contractions} yield the result, by the Banach principle.
\end{proof}

{\bf Remark.} Corollary \ref{random-DS} holds if $(X_n)$ is any centered independent
 sequence (not necessarily iid) with $\sup_{n\ge 1} \|X_n\|_\infty < \infty$, since
\cite[Corollary 4.3]{CL05} applies to it.

\medskip

\subsection{Comments}
In subsection \ref{orbital} we defined the topological entropy of  bounded sequences
$(a_n)_{n\ge 0}$. Note that the entropy  of $(\mu(n))$ does not appear in Sarnak's
 conjecture, nor in the limit theorems in Section \ref{modulated} and in Appendix B; 
they all depend on Davenport's estimate, as do the limit theorems in this section.

The independence on the topological entropy of the sequence is shown  by the fact that
 the topological entropy of the M\"obius sequence is positive \cite[Section 6]{AKLR1}
(already noted by Sarnak \cite{sarnak},\cite{sarnak1}). The topological
 entropy of the Rudin-Shapiro sequence is zero, since   its dynamical system is
uniquely ergodic with spectral multiplicity 2 (e.g. \cite{Q}), which implies zero entropy
by \cite[Corollaries 14.4]{R} (see also \cite[p. 71, Corollary]{P}).
\bigskip

\section*{Acknowledgements}

The work of the first author was  supported   by an  ARC Grant.
Part of this work was carried out during  visits of the second author to the University 
of Rouen - Normandie, to whom the second author is grateful for its
hospitality and support. The second author is also grateful to Jer\^ome Dedecker
for his warm hospitality in Paris.

The authors thank Eli Glasner for fruitful discussions. The second author thanks
Yuri Tomilov for his comments, which led to Proposition \ref{quasi-compact}, and
Guy Cohen for discussions concerning subsection \ref{random}.

\end{document}